\documentclass{article}
\usepackage[latin9]{luainputenc}
\usepackage{geometry}
\usepackage{xcolor}
\usepackage{pdfcolmk}
\usepackage{float}
\usepackage{mathrsfs}
\usepackage{amsmath}
\usepackage{amssymb}
\usepackage{graphicx}
\usepackage{esint}
\PassOptionsToPackage{normalem}{ulem}
\usepackage{ulem}

\makeatletter

\providecolor{lyxadded}{rgb}{0,0,1}
\providecolor{lyxdeleted}{rgb}{1,0,0}
\usepackage{amsfonts}
\usepackage{graphicx}
\usepackage{caption}
\usepackage{subcaption}
\usepackage{amsfonts}
\usepackage{amsthm}

\theoremstyle{definition}
\newtheorem{theorem}{Theorem} 
\newtheorem{corollary}{Corollary} 
\newtheorem{lemma}{Lemma}
\newtheorem{proposition}{Proposition}
\newtheorem{assumption}{Assumption} 
\theoremstyle{remark}
 
\newtheorem{definition}{Definition} 
\newtheorem{remark}{Remark}

\makeatother

\begin{document}

\title{On the Divergence and Vorticity of Vector Ambit Fields}

\author{Orimar Sauri\thanks{This study was funded by the Villum Fonden as part of the project
number 11745 titled {\it''Ambit Fields: Probabilistic Properties and Statistical Inference''}.} \\
 %EndAName
Department of Mathematics and CREATES\\
 Aarhus University\\
 osauri@math.au.dk }

\date{\today }

\maketitle
 
\begin{abstract}
This paper studies the asymptotic behavior of the flux and circulation
of a subclass of random fields within the family of 2-dimensional
vector ambit fields. We show that, under proper normalization, the
flux and the circulation converge stably in distribution to certain
stationary random fields that are defined as line integrals of a L\'evy
basis. A full description of the rates of convergence and the limiting
fields is given in terms of the roughness of the background driving
L\'evy basis and the geometry of the ambit set involved. We further
discuss the connection of our results with the classical Divergence
and Vorticity Theorems. Finally, we introduce a class of models that
are capable to reflect stationarity, isotropy and null divergence
as key properties.	
\end{abstract}

\section{Introduction\label{intro}}

A classical result of vector calculus, namely Stokes' Theorem, allows
to express the \textit{vorticity} (also known as curl) and \textit{divergence}
operators in terms of the \textit{circulation} and the \textit{flux}
of a vector field. In 2 dimensions, it states that 
\begin{align}
\int_{D}\nabla^{\perp}\cdot u(x,y)\mathrm{d}x\mathrm{d}y & =\ointop_{\partial D}u(s)\cdot n_{D}^{\perp}(s)\mathrm{d}s;\label{Classicalvorticity}\\
\int_{D}\nabla\cdot u(x,y)\mathrm{d}x\mathrm{d}y & =\ointop_{\partial D}u(s)\cdot n_{D}(s)\mathrm{d}s,\label{Classicaldivergence}
\end{align}
with $\nabla:=(\partial_{x},\partial_{y})'$ , $\nabla^{\perp}:=(-\partial_{y},\partial_{x})'$,
$u:\mathbb{R}^{2}\rightarrow\mathbb{R}^{2}$ a continuously differentiable
field, $D$ a compact set on $\mathbb{R}^{2}$ with area $\left|D\right|>0$
and smooth boundary $\partial D$, $n_{D}$ the outwards unitary vector
on $\partial D$ and $n_{D}^{\perp}$ is the unitary vector which
is perpendicular (counterclockwise) to $n_{D}$.

The quantities obtained by normalizing the right-hand side of (\ref{Classicalvorticity})
and (\ref{Classicaldivergence}) by $\left|D\right|$ are termed the
2-dimensional mean circulation and mean flux, respectively. In fluid
mechanics, the mean circulation measures the degree of rotation and
the mean flux measures the degree of incompressibility. If we let
$u$ be the 2-dimensional velocity field of a streaming fluid and
$D$ a disk, then the mean circulation and the mean flux will measure
the movement of the fluid along and through the region $D$, respectively.
The more the fluid is aligned to $\partial D$ (the larger the mean
circulation), the more the motion is of rotational type. A large positive
(negative) mean flux describes the situation where more (less) fluid
is entering $D$, which implies that the density of the fluid is increasing
(decreasing). Hence, when the radius of $D$ is small, the mean circulation
and mean flux quantify the pointwise rotation/vorticity and the pointwise
change of density of the fluid, respectively. The concept of incompressibility
expresses the fact that the density of the fluid is constant and this
property is a common assumption in many turbulence studies. Likewise,
vorticity and the related concept of vortex merging is believed to
be a main dynamic process for 2-dimensional turbulent flows. See for
instance \cite{BoffeEcke12} and \cite{RivWuYeun01}.
Having in mind application to turbulence modeling, it is therefore
crucial to understand the behavior of the limits \\

\begin{equation}
\lim_{r\downarrow0}\frac{1}{\pi r^{2}}\ointop_{\partial D}u(s)\cdot n_{D}(s)\mathrm{d}s;\,\,\,\,\,\lim_{r\downarrow0}\frac{1}{\pi r^{2}}\ointop_{\partial D}u(s)\cdot n_{D}^{\perp}(s)\mathrm{d}s.\label{divvorticitylim}
\end{equation}

The main goal of this paper is to study the limits appearing in (\ref{divvorticitylim})
for a certain class of non-smooth random fields belonging to the family
of \textit{ambit fields}. The class of ambit fields was introduced
originally in \cite{BNSch07} as a potential way to study the velocity
field in a turbulent flow. A distinctive characteristic of the ambit
stochastics approach, which distinguishes this from others, is that
it specifically incorporates additional inputs referred to as volatility
or intermittency. Another special feature is the presence of ambit
sets that delineate which part of space-time may influence the value
of the field at any given point in space-time. More specifically,
a random field $(Y_{t}(p))_{t\in\mathbb{R},p\in\mathbb{R}^{d}}$ is
said to be an ambit field if it admits the following dynamics 
\[
Y_{t}(p)=\mu+\int_{A_{t}(p)}F(t,s,p,q)\sigma_{s}(q)L(\mathrm{d}s\mathrm{d}q)+\int_{B_{t}(p)}G(t,s,p,q)\chi_{s}(q)\mathrm{d}s\mathrm{d}q,
\]
where $t$ denotes time while $p$ gives the position in $d$-dimensional
Euclidean space. Further, $A_{t}(p)$ and $B_{t}(p)$ are subsets
of $\mathbb{R}\times\mathbb{R}^{d}$, termed ambit sets, $F$ and
$G$ are deterministic weight functions, and $\sigma$ and $\chi$
are stochastic fields. Finally, $L$ denotes a L\'evy basis (i.e. an
independently scattered and infinitely divisible random measure).
For surveys on ambit fields and their relation to turbulence modeling,
we refer to \cite{HedSch14}, \cite{BNBEnthVeraat15}, \cite{BNBEnthVeraat11} and reference therein.
In this paper, we will focus on purely spatial stationary ambit fields
of the form 
\begin{equation}
Y(p)=\int_{\mathcal{R}+p}F(p-q)V(q)L(\mathrm{d}q),\label{ambiteq1}
\end{equation}

with $F$ a vector-valued function, $\mathcal{R}$ a compact set in
$\mathbb{R}^{2}$, $V$ a real-valued measurable random field, and
$L$ a real-valued homogeneous L\'evy basis.

The null-space version of ambit fields are called \textit{L\'evy semistationary
	processes} ($\mathcal{LSS}$ for short) which are stochastic processes
on a filtered probability space $(\Omega,\mathcal{F},(\mathcal{F}_{t})_{t\in\mathbb{R}},\mathbb{P})$
that are described by the formula 
\[
Y_{t}=\theta+\int_{-\infty}^{t}g(t-s)\sigma_{s}\mathrm{d}L_{s}+\int_{-\infty}^{t}q(t-s)a_{s}\mathrm{d}s,\qquad t\in\mathbb{R},
\]
where $\theta\in\mathbb{R}$, $L$ is a L\'evy process, $g$ and $q$
are deterministic functions such that $g(x)=q(x)=0$ for $x\leq0$,
and $\sigma$ and $a$ are adapted processes. When $L$ is a two-sided
Brownian motion, $Y$ is called a \textit{Brownian semistationary
	process} ($\mathcal{BSS}$). As further references to theory and applications
of $\mathcal{LSS}$, see for instance  \cite{Pakkanen11},
\cite{BNBEnthVeraat13} and \cite{PedSau15} and therein references. For recent results
on limit theorems see \cite{CorcueraHedPakkPod13} and \cite{BasseHeinrPod17}.

The relations (\ref{Classicalvorticity}) and (\ref{Classicaldivergence})
are particular cases of Stokes' Theorem which, in its standard form,
is stated for differentiable forms over smooth manifolds. We refer
to \cite{Katz79} for an interesting historical review on this topic.
Several extensions of Stokes' Theorem can be found in the literature,
mainly those involving surface/line integrals of smooth forms over
non-smooth regions, e.g. fractals or paths of stochastic processes.
See for instance \cite{Hsu2002}, \cite{IkedaManabe79} and references
therein. A non-stochastic approach is described in \cite{Harrison99}, where the authors proved a version
of Stokes' Theorem for non-smooth manifolds. This is done by introducing
a certain type of surface/line integral of smooth forms over what
is called \textit{chainlets}, which turned out to be a general class
of regions that contains, among others, smooth sub-manifolds, fractals
and vector fields. In contrast, very little has been done in the other
direction, i.e. to consider line/surface integrals of non-smooth forms
over smooth manifolds. This paper intends to develop some results
in that direction. To our knowledge, the only existing work in relation
to non-smooth forms is \cite{Zust11}, in which the author, by employing
Young's approach (see \cite{Young36}), introduced an integral for
non-smooth forms over Lipschitz manifolds. However, the stated version
of Stokes' Theorem requires the form to be constant.

The organization of the present work is as follows: In Section 2 we
introduce the basic notations as well as the basic assumptions. We
also recall several results and concepts related to stable convergence
of r.v.s, L\'evy bases and infinite divisibility. We further give some geometrical preliminaries. Our
main results, concerning the asymptotic behavior of the flux and circulation,
is stated in Section 3. Specifically, we show that under proper normalization,
the flux and the circulation of a purely spatial stationary ambit
field converge stably to certain random fields that are defined in
terms of a separable L\'evy basis whose control measure is the 1-dimensional
Hausdorff measure. We postpone their proof to Section 5. As an application of our
results, we introduce in Section 4 a class of purely spatial and $\mathbb{R}^{2}$-valued
ambit fields which have stationary and isotropic increments. Such
a family of fields was originally introduced jointly with Ole E. Barndorff-Nielsen
and J\"urgen Schmiegel as a potential modeling framework for 2-dimensional
turbulent flows. Moreover, these fields are rotational and have the
property of incompressibility. We also include two appendixes. Appendix A provides
a Steiner-type formula for closed sets, and briefly describes the convergence of stochastic integrals with respect to L\'evy bases. Appendix B focus on technical results that are used in the proof of our main results.

\section{Preliminaries and basic notation\label{sec:Preliminaries-and-basic}}

This part is devoted to introduce the basic notations as well as to
recall several basic results and concepts that will be used through
this paper.

\subsection{Stable convergence\label{Preliminariesstableconv}}

For the rest of this paper we will consider $\left(\Omega,\mathcal{F},\mathbb{P}\right)$
to be a complete probability space. As usual, the
notation $\overset{\mathbb{P}}{\rightarrow}$ means convergence in
probability and the notation $X_{n}=o_{\mathbb{P}}(Y_{n})$ means
that $X_{n}/Y_{n}$$\overset{\mathbb{P}}{\rightarrow}$0 when $n\rightarrow\infty$.
Given a sub-$\sigma$-field $\mathcal{G}\subseteq\mathcal{F}$ and
a sequence of random variables (r.v.'s
for short) $(\xi_{n})_{n\geq1}$ on $\left(\Omega,\mathcal{F},\mathbb{P}\right)$,
the notation $\xi_{n}\overset{\mathcal{G}\text{-}d}{\longrightarrow}\xi$
will mean that as $n\rightarrow \infty$, $\xi_{n}$ converges $\mathcal{G}$-stably in distribution
towards a random variable (r.v. for short) $\xi$ (defined possibly
on an extension of $\left(\Omega,\mathcal{F},\mathbb{P}\right)$),
that is, for any $F\in\mathcal{G}$, with $\mathbb{\mathbb{P}}(F)>0$,
conditioned on the event $F$, $\xi_{n}$$\rightarrow\xi$, weakly.
In the same framework, if $(X_{n}(p))_{p\in\mathbb{R}^{d},n\in\mathbb{N}}$
is a family of random fields, we will write $X_{n}\overset{\mathcal{G}\text{-}fd}{\longrightarrow}X$
if the finite-dimensional distributions (f.d.d. for short) of $X_{r}$
converge $\mathcal{G}$-stably toward the f.d.d. of $X$. We refer
the reader to \cite{HauslerLuschgy15} for a concise exposition of
stable convergence.

\subsection{L\'evy bases and infinite divisibility\label{PreliminariesLBID}}

The symbols $D_{r}(p)$ and  $r\mathbb{S}^{d-1}(p)$ will denote the closed disk and the sphere with center $p$
and radius $r$. When $p=0$, we will just write $D_{r}$ and $r\mathbb{S}^{d-1}$ instead
of $D_{r}(0)$ and $r\mathbb{S}^{d-1}(p)$, respectively. For any $A\subseteq\mathbb{R}^{d},$ we let $-A=\{-x:x\in A\}$.
Furthermore, we denote by $\mathring{A},\overline{A},\partial A\text{ and }A^{c}$
the interior, the closure, the boundary and the complement of $A$,
respectively and we put $A^{*}=\overline{A^{c}}$. The inner product
and the norm of vectors $x,y\in\mathbb{R}^{d}$ will be represented
by $x\cdot y$ and $\left\Vert x\right\Vert $, respectively. Let
$\mu$ be a measure on $\mathcal{B}(\mathbb{R}^{d})$, the Borel sets
on $\mathbb{R}^{d}$, and let $\mathcal{B}_{b}^{\mu}(\mathbb{R}^{d}):=\{A\in\mathcal{B}(\mathbb{R}^{d}):\mu(A)<\infty\}.$
The family $L=\{L\left(A\right):A\in\mathcal{B}_{b}^{\mu}(\mathbb{R}^{d})\}$
of real-valued r.v.'s will be called a \textit{L\'evy basis} if it is
an infinitely divisible (ID for short) independently scattered random
measure, that is, $L$ is $\sigma$-additive almost surely and such
that for any $A,B\in\mathcal{B}_{b}^{\mu}(\mathbb{R}^{d})$, $L(A)$
and $L(B)$ are ID r.v.'s that are independent whenever $A\cap B=\emptyset$.
The cumulant of a r.v. $\xi$, in case it exists, will be denoted
by $\mathcal{C}(z\ddagger\xi):=\log\mathbb{E}(e^{iu\xi})$. We will
say that $L$ is \textit{separable} with \textit{control measure}
$\mu$, if 
\[
\mathcal{C}(z\ddagger L(A))=\mu(A)\psi(z),\,\,\,A\in\mathcal{B}_{b}^{\mu}(\mathbb{R}^{d}),z\in\mathbb{R},
\]
where 
\[
\psi(z):=i\gamma z-\frac{1}{2}b^{2}z^{2}+\int_{\mathbb{R}\backslash\{0\}}(e^{izx}-1-izx\mathbf{1}_{\left|x\right|\leq1})\nu(\mathrm{d}x),\,\,\,z\in\mathbb{R},
\]
with $\gamma\in\mathbb{R},$ $b\geq0$ and $\nu$ is a L\'evy measure,
i.e. $\nu(\{0\})=0$ and $\int_{\mathbb{R}\backslash\{0\}}(1\land\left|x\right|^{2})\nu(\mathrm{d}x)<\infty$.
When $\mu=Leb$, in which $Leb$ represents the Lebesgue measure on
$\mathbb{R}^{d}$, $L$ is called \textit{homogeneous}. The ID r.v.
associated to the characteristic triplet $\left(\gamma,b,\nu\right)$
is called the \textit{L\'evy seed} of $L$ and will be denoted by $L'$.
As usual, $\left(\gamma,b,\nu\right)$ will be called the characteristic
triplet of $L$ and $\psi$ its characteristic exponent. In this paper,
the sigma field generated by $L$ is denoted by $\mathcal{F}_{L}$.

For any L\'evy measure $\nu$, we associate the functions $\nu^{\pm}:(0,1)\rightarrow\mathbb{R}^{\text{+}}$,
defined as $\nu^{+}(x):=\nu(x,\infty)$ and $\nu^{-}(x):=\nu(-\infty,-x).$
Let $K_{+},K_{-}\geq0$ and $0<\beta\leq2$. A separable L\'evy basis is called
strictly $\beta$-stable with parameters $(K_{+},K_{-},\beta,\gamma)$
if its L\'evy seed is distributed according to a strictly $\beta$-stable
distribution, that is, $L'$ is Gaussian if $\beta=2$, while for $\beta<2$ the characteristic triplet of $L'$ has no
Gaussian component ($b=0$), its L\'evy measure satisfies 
\[
\frac{\nu(\mathrm{d}x)}{\mathrm{d}x}=K_{+}\left|x\right|^{-1-\beta}\mathbf{1}_{\{x>0\}}+K_{-}\left|x\right|^{-1-\beta}\mathbf{1}_{\{x<0\}},
\]
and $\gamma=(K_{+}-K_{-})/(1-\beta)$ if $\beta\neq1$, and $\gamma$
arbitrary with $K_{+}=K_{-}$when $\beta=1$.

\subsection{Geometrical preliminaries\label{PreliminariesJCHausdorff}}
Fix $A\subseteq\mathbb{R}^{d}$ a closed set and denote by $\mathcal{H}^{n}$ the $n$th-dimensional Hausdorff measure. The \textit{normal cone} of $A$ at $p\in A$ is defined as 
\[
\mathrm{nor}(A,p):=\{u\in\mathbb{R}^{d}:u\cdot v\leq0,\,v\in\mathrm{Tan}(A,p)\},
\]
where $\mathrm{Tan}(A,p)$ denotes the set of all \textit{tangent
	vectors} to $A$ at $p$, that is, $v\in\mathrm{Tan}(A,p)$ if and
only if there is a sequence $(p_{n})\subseteq A\backslash\{p\}$ such
that $p_{n}\rightarrow p$ and $\frac{p_{n}-p}{\left\Vert p_{n}-p\right\Vert }\rightarrow\frac{v}{\left\Vert v\right\Vert }$,
as $n\rightarrow\infty$.
We recall that a \textit{ Jordan curve} $\emptyset\neq C\subset\mathbb{R}^{d}$
is a curve in $\mathbb{R}^{d}$ parametrized by $\varphi:[0,1]\rightarrow C$,
such that $\varphi$ is continuous and injective on $(0,1)$ and $\varphi(0)=\varphi(1)$.
We will say that a compact set $A\subseteq\mathbb{R}^{2}$ is a \textit{Jordan
	domain} if $\mathring{A}$ is totally connected and $\partial A\neq\emptyset$
is a Jordan curve. In this framework, we say that a Jordan domain
$A$ has \textit{Lipschitz-regular} boundary if the parametrization of
$\partial A$ is Lipschitz and for $\mathcal{H}^{1}$-a.a. $q\in\partial A$ there is $u_A(q)\in\mathbb{S}^{d-1}$ orthogonal to $\mathrm{Tan}(\partial A,q)$, such that
\[ \mathrm{nor}(A,q)=\{\lambda u_{A}(q):\lambda\geq0\};\,\,\,\mathrm{nor}(A^{*},p):=\{-\lambda u_{A}(q):\lambda\geq0\}.\] 
In other words, a Jordain domain with Lipschitz boundary is regular if it has unique outwards and inwards unit vectors almost everywhere.

The \textit{metric projection} on $A$,
$\Pi_{A}:\mathbb{R}^{d}\rightarrow A$, is the set function
\[
\Pi_{A}(q):=\{p\in A:d_{A}(q)=\left\Vert p-q\right\Vert \},
\]
where $d_{A}(q):=\inf_{p\in A}\left\Vert p-q\right\Vert$. We set 
\[
\mathrm{Unp}A:=\{q\in\mathbb{R}^{d}:\exists!p\in A\text{ s.t. }d_{A}(q)=\left\Vert p-q\right\Vert \}.
\]
Under the previous notation, the {\it reduced normal bundle} and the {\it reach function} of $A$ are  given, respectively, by 
\[
N(A)=\{(\Pi_{A}(q),\frac{q-\Pi_{A}(q)}{\left\Vert q-\Pi_{A}(q)\right\Vert }):q\in\mathrm{Unp}(A)\backslash A\},
\]
and, $\delta_A(q,u):=0$ for $(q,u)\in N(A)^c$ and for $(q,u)\in N(A)$
\[
\delta_A(q,u):=\inf\{t\geq0:q+tu\in\mathrm{Unp}(A)^{c}\}.
\]
For $r\geq0$, the $r$-\textit{parallel set} of $A$ is defined as
\[
A_{\oplus r}:=\{q\in\mathbb{R}^{d}:d_{A}(q)\leq r\}.
\]
\section{Divergence and Vorticity Theorems for Ambit fields}

For the rest of this section we will be interested in the asymptotic
behavior of the following functionals
\begin{align}
\mathscr{C}_{r}(p;X):= & \ointop_{r\mathbb{S}^1(p)}X(q)\cdot \mathrm{d}q=r\int_{0}^{2\pi}X(p+ru(\theta))\cdot u^{\perp}(\theta)\mathrm{d}\theta,\,\,\,p\in\mathbb{R}^{2},r>0,\label{generalizedcirculation}\\
\mathcal{\mathscr{D}}_{r}(p;X):= & \ointop_{r\mathbb{S}^1(p)}X(s)\cdot n(s)\mathrm{d}s=r\int_{0}^{2\pi}X(p+ru(\theta))\cdot u(\theta)\mathrm{d}\theta,\,\,\,p\in\mathbb{R}^{2},r>0,\label{eq:generalizedflux}
\end{align}
as $r\downarrow0$. Above $X$ is a vector-valued random field, $n$
is the outward unit vector in $r\mathbb{S}^1(p)$, that is, $n(p+ru(\theta))=u(\theta)$,
with $u(\theta):=(\cos(\theta),\sin(\theta))'$ for $0\leq\theta\leq2\pi$,
and $(x,y)^{\perp}=(-y,x)$.

When the mapping $p\mapsto X(p)$ is smooth almost surely, by the
usual Stokes' Theorem, it holds that
\begin{align*}
\lim_{r\downarrow0}\frac{1}{2\pi r^{2}}\mathscr{C}_{r}(p;X)\stackrel{a.s.}{\rightarrow}\nabla^{\perp}\cdot X(p),\,\,\,p & \in\mathbb{R}^{2};\\
\lim_{r\downarrow0}\frac{1}{2\pi r^{2}}\mathscr{D}_{r}(p;X)\stackrel{a.s.}{\rightarrow}\nabla\cdot X(p),\,\,\,p & \in\mathbb{R}^{2}.
\end{align*}
where $\nabla:=(\partial_{x},\partial_{y})'$ and $\nabla^{\perp}:=(-\partial_{y},\partial_{x})'$.
However, not surprisingly, such a result does not hold anymore when
one consider fields of the form of (\ref{ambiteq1}). Furthermore,
as expected, when the kernel $F$ is smooth enough, the rates of convergence
for $\mathscr{C}_{r}$ and $\mathcal{\mathscr{D}}_{r}$ depend entirely
on the ambit set and background driving L\'evy basis. Before presenting
our main results we first explain the intuition behind them.

\subsection{Some intuitive description\label{subsec:Some-intuitive-description}}

Previously, we mentioned that when the kernel involved in the definition
of (\ref{ambiteq1}) is smooth, then the asymptotic behavior of $\mathscr{C}_{r}$
and $\mathcal{\mathscr{D}}_{r}$ would be determined by the background
driving L\'evy basis and the ambit set. In this subsection we will give
an intuitive description of why this would be the case. As a motivation
and starting point in our analysis, let us first describe what is
known in the one-dimensional case. For $s_{0}>0$ and $t\in\mathbb{R},$
let $\mathcal{R}(t):=[-s_{0}+t,t]$ and put
\[
X_{t}:=\int_{\mathcal{R}(t)}f(t-s)\mathrm{d}L_{s}.
\]
where $L$ denotes a L\'evy process on $\mathbb{R}$ with characteristic
triplet $(\gamma,b,\nu)$, and $f$ a real-valued function. In this
case for $r>0$, we get that
\begin{equation}
\ointop_{r\mathbb{S}^0(t)}X_{s}\cdot\mathrm{d}s  =X_{t+r}-X_{t-r}=\ointop_{r\mathbb{S}^0(t)}\partial X_{s}\cdot\mathrm{d}s+\ointop_{r\mathbb{S}^0(t)}\mathring{X}_{s}\cdot\mathrm{d}s,\label{eq:decompositionint1}
\end{equation}
where
\[ \partial X_{t}:=  f(0)L_{t}-f(s_{0})L_{t-s_{0}};\,\,\,\mathring{X}_{t}:=  X_{t}-\partial X_{t}. \]
The notation $\partial X,\mathring{X}$ is not by chance, many properties
of $\partial X\text{ and }\mathring{X}$ are completely determined
by the interaction of $f$ and $L$ on $\partial\mathcal{R}(t)$ and
$\mathring{\mathcal{R}}(t)$, respectively. We first observe that
$\mathring{X}$ admits the representation 
\[
\mathring{X}_{t}=\int_{\mathcal{R}(t)}g(t-s)\mathrm{d}L_{s},
\]
where $g$ is absolutely continuous and with $\left.g(\cdot)\right|_{-\partial\mathcal{R}}\equiv0$,
and that
\[
\ointop_{r\mathbb{S}^0(t)}\partial X_{s}\cdot\mathrm{d}s=\int_{\left(\partial\mathcal{R}(t)\right)_{\oplus r}}h(t-s)\mathrm{d}L_{s},
\]
for some measurable function $h$. Additionally, we have that if $f$
is continuously differentiable, then the path $t\mapsto\mathring{X}_{t}$
is almost surely absolutely continuous (see \cite{BasseRosinski13},
cf. \cite{BravermanSanorod98} and \cite{BassePed09} for more details)
in such a way that
\begin{equation}
\frac{1}{2r}\ointop_{r\mathbb{S}^0(t)}\mathring{X}_{s}\cdot\mathrm{d}s\overset{\mathbb{P}}{\rightarrow}\int_{\mathcal{R}(t)}f'(t-s)\mathrm{d}L_{s},\,\,\,t\geq0.\label{limitinterior}
\end{equation}

On the other hand, $\ointop_{r\mathbb{S}^0(t)}\partial X_{s}\cdot\mathrm{d}s$
consists of the increments of size $r>0$ of $L$ around $\partial\mathcal{R}(t)$.
Therefore, under proper normalization, $\ointop_{r\mathbb{S}^0(t)}\partial X_{s}\cdot\mathrm{d}s$
has a non-trivial limit if and only if the same property holds for
the increments of $L$. In connection to the former, it is well known
that when $b>0$, the increments of $L$ are totally dominated by
its Gaussian component. Moreover, if $L$ is of bounded variation,
then the increments of $L$ are totally dominated by the drift component.
For these facts we refer to \cite[p.~16]{Bertoin98}. Ultimately,
when $b=0$ and $L$ is of unbounded variation, typically the increments
of $L$ are in the domain of attraction of a strictly $\beta$-stable
distribution with $L_{t+r}-L_{t-r}=O_{\mathbb{P}}(r^{1/\beta})$, for
some $1\leq\beta<2$. All in all then give us the following asymptotics for $\ointop_{r\mathbb{S}^0(t)}X_{s}\cdot\mathrm{d}s$
\begin{enumerate}
	\item \textbf{Gaussian regime}: If $\left.f\right|_{-\partial\mathcal{R}}\neq0$
	and $b>0$, then $\ointop_{r\mathbb{S}^0(t)}X_{s}\cdot\mathrm{d}s=O_{\mathbb{P}}(r^{1/2})$
	with Gaussian limit.
	\item \textbf{Stable regime}: If $\left.f\right|_{-\partial\mathcal{R}}\neq0$,
	$b=0$ , $L$ is of unbounded variation, then $\ointop_{r\mathbb{S}^0(t)}X_{s}\cdot\mathrm{d}s=O_{\mathbb{P}}(r^{1/\beta})$
	with strictly $\beta$-stable limit, for some $1\leq\beta<2$.
	\item \textbf{``Classical'' regime:} If $L$ is of bounded variation or
	if $\left.f\right|_{-\partial\mathcal{R}}=0$ with arbitrary $L$,
	then $\ointop_{r\mathbb{S}^0(t)}X_{s}\cdot\mathrm{d}s=O_{\mathbb{P}}(r)$.
\end{enumerate}

As a final remark for the one-dimensional framework, we note that
when $f$ is not smooth enough, the regimes previously stated are
not valid anymore. For this situation, we refer to \cite{CorcueraHedPakkPod13},
\cite{BasseLachPod17} and references therein.

Now, when we consider the ID field given by 
\begin{equation}
X(p):=\int_{\mathcal{R}+p}F(p-q)L(\mathrm{d}q).\label{eq:ambitfield}
\end{equation}
with $F$ a vector valued function, $\mathcal{R}$ a compact set in
$\mathbb{R}^{2}$ and $L$ a real-valued homogeneous L\'evy basis, we
may in principle try to follow the same reasoning as in the temporal
case and expect to recover similar results. Thus, we may try to
decompose $X$ as
\begin{equation}
X(p)=\partial X(p)+\mathring{X}(p),\label{eq:formalrepresentation}
\end{equation}
for some fields $\partial X$ and $\mathring{X}$ whose trajectories
are totally determined by the behavior of $L$ and $F$ on $\partial\mathcal{R}(p)$
and $\mathring{\mathcal{R}}(p)$, respectively. Unfortunately, to
our knowledge, there is no a general identification of the fields
$\partial X$ and $\mathring{X}$ appearing in (\ref{eq:formalrepresentation}),
except in very special cases (see for instance \cite{CairoliWalsh75}). However, our proof
is based on an analogous decomposition to that in (\ref{eq:decompositionint1}).
Specifically, we decompose 
\begin{equation}
\ointop_{r\mathbb{S}^1(p)}X\cdot n\mathrm{d}s=\int_{\left(\partial\mathcal{R}(p)\right)_{\oplus r}}H(p-q)L(\mathrm{d}q)+E_r(p),\label{decompositiondiv}
\end{equation}
for a certain smooth function $H$ and $E_r(p)$ a random field satisfying
that as $r\downarrow 0$
\begin{align*}
\frac{1}{\left|D_{r}(p)\right|}E(p) & \overset{\mathbb{P}}{\rightarrow}\int_{\mathcal{R}(p)}\nabla\cdot F(p-q)L(\mathrm{d}q).
\end{align*}
Therefore, as in the one-dimensional case, understanding the asymptotic
behavior of $\ointop_{r\mathbb{S}^1(p)}X\cdot n\mathrm{d}s$ requires
a full knowledge of the asymptotic behavior of $L$ on the $r$-parallel
sets of $\mathcal{R}(p)$.

\subsection{The divergence and vorticity theorems}

In this part we present our main results about $\mathscr{C}_{r}$
and $\mathcal{\mathscr{D}}_{r}$ for the stationary ID field
\begin{equation}
X(p):=\int_{\mathcal{R}+p}F(p-q)L(\mathrm{d}q),\,\,\,p\in\mathbb{R}^{2}.\label{IDXdef}
\end{equation}
with $F$ continuously differentiable on $-\mathcal{R}$ and $L$
a real-valued homogeneous L\'evy basis with characteristic triplet $\left(\gamma,b,\nu\right)$.
In the next subsection we will study the case in which an additional
stochastic field $V$ is included in (\ref{IDXdef}).

For the rest of this paper we will be working under the following
assumption on the ambit set $\mathcal{R}$.

\begin{assumption}\label{assumptionambitset1}The ambit set $\mathcal{R}\subseteq\mathbb{R}^{2}$
	can be written as 
	\[
	\mathcal{R}=\mathcal{R}_{1}\backslash\bigcup_{i=2}^{n}\mathcal{R}_{i},
	\]
	where $\mathcal{R}_{1},\ldots,\mathcal{R}_{n}$ are Jordan
	domains with Lipschitz-regular boundary satisfying that $\mathcal{R}_{i}\cap\mathcal{R}_{j}=\emptyset$ and  $\mathcal{R}_{i}\subset\mathcal{R}_{1}$
	for $i,j=2,\ldots,n$ and $i\neq j$. Furthermore, for $i=1,\ldots,n$
	\[
	\int_{\mathbb{S}^{1}}\left\lbrace  \sum_{q:(q,u)\in N(\mathcal{R}_{i})}(\delta_{\mathcal{R}_{i}}(q,u)\wedge1)+\sum_{q:(q,u)\in N(\mathcal{R}_{i}^{*})}(\delta_{\mathcal{R}_{i}^{*}}(q,u)\wedge1)\right\rbrace\mathcal{H}^{1}(\mathrm{d}u)<\infty.
	\]
\end{assumption}
Some examples of sets satisfying Assumption \ref{assumptionambitset1} will be presented in the next section. However, the type of sets we have in mind can be visualized in Figure \ref{ambitsetR}. 
\begin{figure}[h]
	\center\includegraphics[width=0.75\textwidth]{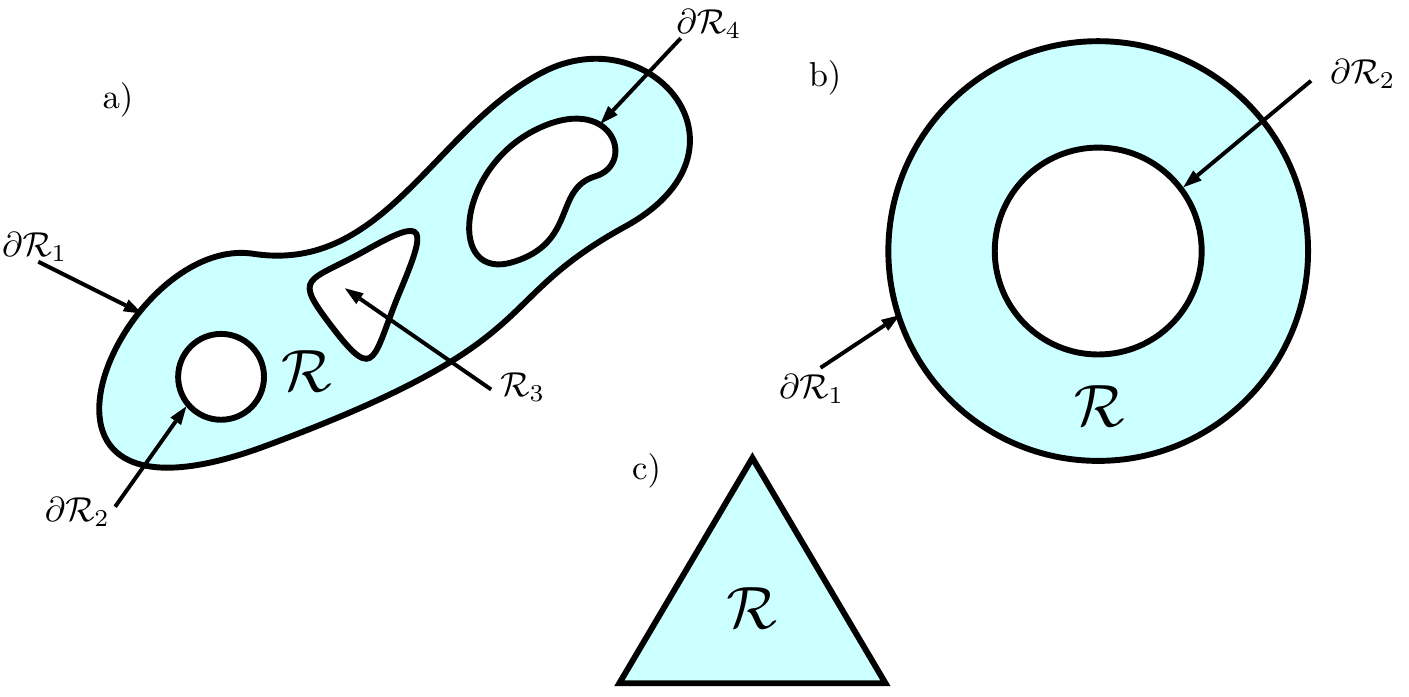}\caption{\label{ambitsetR}Typical examples of the type of ambit sets considered
		in Assumption \ref{assumptionambitset1}.}
\end{figure}

\begin{remark}\label{remarkambitset} Assumption \ref{assumptionambitset1} allows to extended the definition of $u_{\mathcal{R}}$ (the outward unit vector) to the whole $\partial\mathcal{R}$ by letting $u_{\mathcal{R}}(q)\equiv0$ in the irregular points of $\partial\mathcal{R}(p)$. Furthermore, since $\mathcal{R}(p)$ is just a translation of $\mathcal{R}$,
	we get that $\mathcal{R}(p)$ satisfies Assumption \ref{assumptionambitset1}
	if and only if $\mathcal{R}$ does and in this case $u_{\mathcal{R}(p)}(q+p)=u_{\mathcal{R}}(q)$
	for $q\in\partial\mathcal{R}$.
\end{remark}

As in Subsection \ref{subsec:Some-intuitive-description},
our analysis will be divided into three different scenarios. For the sake of exposition, all the proofs
of this section will be postponed to Section \ref{sec:Proofs}. Let
us start with the case in which the Gaussian part of $L$ dominates
the asymptotics. In order to improve the presentation of this result
let us first introduce the limiting fields. Recall the
notation $(x,y)^{\perp}=(-y,x).$ Given $b>0$, $F$ as above and
$\mathcal{R}$ as in Assumption \ref{assumptionambitset1}, the fields
$\{\mathscr{C}_{W}(p;F,\mathcal{R}),\mathscr{D}_{W}(p;F,\mathcal{R})\}_{p\in\mathbb{R}^{2}}$
will denote a collection of stationary Gaussian fields defined for
any $p\in\mathbb{R}^{2}$ as
\begin{align*}
\mathscr{D}_{W}(p;F,\mathcal{R}):= & \int_{\partial\mathcal{R}(p)}F(p-q)\cdot u_{\mathcal{R}(p)}(q)W_{\mathcal{H}^{1}}(\mathrm{d}q),\\
\mathscr{C}_{W}(p;F,\mathcal{R}):= & \int_{\partial\mathcal{R}(p)}F(p-q)\cdot u_{\mathcal{R}(p)}^{\perp}(q)W_{\mathcal{H}^{1}}(\mathrm{d}q),
\end{align*}
where $u_{\mathcal{R}(p)}$ is as in Remark \ref{remarkambitset}.
Furthermore, $W_{\mathcal{H}^{1}}$ is a separable Gaussian L\'evy basis
(see Subsection \ref{PreliminariesLBID}) defined on an extension
of $\left(\Omega,\mathcal{F},\mathbb{P}\right)$ having the following
properties: 1) Its L\'evy seed has a centered Gaussian distribution
with variance $b^{2}$; 2) Its control measure is $\mathcal{H}^{1}$,
the 1-dimensional Hausdorff measure; 3) $W_{\mathcal{H}^{1}}$ is
independent of $L$.

In this setting, using the notation $v_{\beta}:=2\{\int_{-1}^{1}(1-s^{2})^{\beta/2}\mathrm{d}s\}^{1/\beta}$,
we have that:

\begin{theorem}[{\bf Gaussian attractor}]\label{gaussianattractor}
	Let $\mathcal{R}\subset\mathbb{R}^{2}$ be as in Assumption \ref{assumptionambitset1}.
	Consider $X$ as in (\ref{IDXdef}) with $\left.F\right|_{-\mathcal{\partial R}}\neq0$.
	If $b>0$, then as $r\downarrow0$
	\begin{align*}
	\frac{1}{v_{2}r^{1+1/2}}\mathscr{C}_{r}(p;X) & \overset{\mathcal{F}\text{-}fd}{\longrightarrow}\mathscr{C}_{W}(p;b,F,\mathcal{R});\\
	\frac{1}{v_{2}r^{1+1/2}}\mathscr{D}_{r}(p;X) & \overset{\mathcal{F}\text{-}fd}{\longrightarrow}\mathscr{D}_{W}(p;b,F,\mathcal{R}).
	\end{align*}
	
\end{theorem}

On the other hand, as expected, when $L$ is of bounded variation
the rate of convergence for $\mathscr{C}_{r}$ and $\mathscr{D}_{r}$
are the classical ones, i.e. of order $r^{2}$. More precisely:

\begin{theorem}[{\bf "Classical" regime}]\label{divergenceandvorticitytheorem}
	Let $\mathcal{R}\subset\mathbb{R}^{2}$ be as in Assumption \ref{assumptionambitset1}
	and $X$ as in (\ref{IDXdef}). Then the following convergence holds
	as $r\downarrow0$
	\begin{align*}
	\frac{1}{\pi r^{2}}\mathscr{C}_{r}(p;X) & \stackrel{\mathbb{P}}{\rightarrow}\omega(p),\,\,\,p\in\mathbb{R}^{2};\,\,\,\frac{1}{\pi r^{2}}\mathscr{D}_{r}(p;X)\stackrel{\mathbb{P}}{\rightarrow}\sigma(p),\,\,\,p\in\mathbb{R}^{2},
	\end{align*}
	if one of the following (not-necessarily mutually exclusive) cases
	holds:
	\begin{description}
		\item [{i)}] $b=0$ and $\int_{\mathbb{R}}(1\land\left|x\right|)\nu(\mathrm{d}x)<\infty$;
		\item [{ii)}] $\left.F\right|_{-\mathcal{\partial R}}\equiv0$;
	\end{description}
	where the limiting processes are defined as
	\begin{align*}
	\omega(p) & :=\int_{\mathcal{R}+p}\nabla^{\perp}\cdot F(p-q)\tilde{L}(\mathrm{d}q),\\
	\sigma(p) & :=\int_{\mathcal{R}+p}\mathrm{\nabla}\cdot F(p-q)\tilde{L}(\mathrm{d}q).
	\end{align*}
	with $\nabla:=(\partial_{x},\partial_{y})'$, $\nabla^{\perp}:=(-\partial_{y},\partial_{x})'$,
	$\tilde{L}=L-\gamma_{d}Leb$, where $\gamma_{d}=\gamma-\int_{\left|x\right|\leq1}x\nu(\mathrm{d}x)$
	for case i) while $\gamma_{d}=\gamma$ for case ii) and iii).
	
\end{theorem}

There is another situation in which the classical rate appears. Before presenting this last case we introduce the limiting
fields: For $0<\beta<2$, let $(K_{+},K_{-},\beta,\hat{\gamma})$
be the parameters of a strictly $\beta$-stable r.v. (see Subsection
\ref{PreliminariesLBID}), $F$ as above and $\mathcal{R}$ as in
Assumption \ref{assumptionambitset1}. The family $\{\mathscr{C}_{\beta}(p;F,\mathcal{R}),\mathscr{D}_{\beta}(p;F,\mathcal{R})\}_{p\in\mathbb{R}^{2}}$
will denote two stationary strictly $\beta$-stable fields given for
any $p\in\mathbb{R}^{2}$ as 
\begin{align*}
\mathscr{C}_{\beta}(p;F,\mathcal{R}):=- & \int_{\partial\mathcal{R}(p)}F(p-q)\cdot u_{\mathcal{R}(p)}^{\perp}(q)M_{\mathcal{H}^{1}}^{\beta}(\mathrm{d}q);\\
\mathscr{D}_{\beta}(p;F,\mathcal{R}):=- & \int_{\partial\mathcal{R}(p)}F(p-q)\cdot u_{\mathcal{R}(p)}(q)M_{\mathcal{H}^{1}}^{\beta}(\mathrm{d}q),
\end{align*}
where $u_{\mathcal{R}(p)}$ is as in Remark \ref{remarkambitset},
and $M_{\mathcal{H}^{1}}^{\beta}$ is a separable and strictly $\beta$-stable
L\'evy basis with parameters $(K_{+},K_{-},\beta,\hat{\gamma})$ defined
on an extension of $\left(\Omega,\mathcal{F},\mathbb{P}\right)$ and
whose control measure is the 1-dimensional Hausdorff measure. Moreover,
$M_{\mathcal{H}^{1}}^{\beta}$ is independent of $L$.

\begin{theorem}[{\bf Stable attractor}]\label{stableattractor} Let
	$\mathcal{R}\subset\mathbb{R}^{2}$ be as in Assumption \ref{assumptionambitset1} and $X$ as in (\ref{IDXdef}) with $\left.F\right|_{-\mathcal{\partial R}}\neq0$,
	$b=0$ and $\int_{\mathbb{R}}(1\land\left|x\right|)\nu(\mathrm{d}x)=+\infty$.
	Suppose that there exists $1\leq\beta<2$ such that (see Subsection
	\ref{PreliminariesLBID}) $\nu^{\pm}(x)\sim\tilde{K}_{\pm}x^{-\beta}$
	as $x\downarrow0$ with $\tilde{K}_{+}+\tilde{K}_{-}>0$. Then
	\begin{description}
		\item [{i.}] If $1<\beta<2$, then by letting $K_{\pm}=\beta\tilde{K}_{\pm}$
		we have that for any $p\in\mathbb{R}^{2}$, as $r\downarrow0$
		\begin{align*}
		\frac{1}{v_{\beta}r^{1+1/\beta}}\mathscr{C}_{r}(p;X) & \overset{\mathcal{F}\text{-}d}{\longrightarrow}\mathscr{C}_{\beta}(p;F,\mathcal{R}),\,\,\,p\in\mathbb{R}^{2};\\
		\frac{1}{v_{\beta}r^{1+1/\beta}}\mathscr{D}_{r}(p;X) & \overset{\mathcal{F}\text{-}d}{\longrightarrow}\mathscr{D}_{\beta}(p;F,\mathcal{R}),\,\,\,p\in\mathbb{R}^{2},
		\end{align*}
		\item [{ii.}] If $\beta=1$ in addition assume that $\tilde{K}_{+}=\tilde{K}_{-}$
		and $\mathrm{PV}\int_{-1}^{1}x\nu(\mathrm{d}x)$, the Cauchy principal
		value\footnote{Recall that the Cauchy principal value of an integral around $0$
			is defined as the limit (in case it exists)
			\[
			\mathrm{PV}\int_{-1}^{1}f(x)\nu(\mathrm{d}x):=\lim_{a\downarrow0}\left[\int_{-1}^{-a}f(x)\nu(\mathrm{d}x)+\int_{a}^{1}f(x)\nu(\mathrm{d}x)\right].
			\]
		}, exists. Then, for any $p\in\mathbb{R}^{2}$, as $r\downarrow0$
		\begin{align*}
		\frac{1}{\pi r^{2}}\mathscr{C}_{r}(p;X) & \overset{\mathcal{F}\text{-}d}{\longrightarrow}\omega(p)+\mathscr{C}_{\beta}(p;F,\mathcal{R}),\,\,\,p\in\mathbb{R}^{2};\\
		\frac{1}{\pi r^{2}}\mathscr{D}_{r}(p;X) & \overset{\mathcal{F}\text{-}d}{\longrightarrow}\sigma(p)+\mathscr{D}_{\beta}(p;F,\mathcal{R}),\,\,\,p\in\mathbb{R}^{2}.
		\end{align*}
		where $\hat{\gamma}=\gamma-\mathrm{PV}\int_{-1}^{1}x\nu(\mathrm{d}x)$,
		and $\omega$ and $\sigma$ as in Theorem \ref{divergenceandvorticitytheorem}.
	\end{description}
\end{theorem}

We proceed now to make some remarks about the previous theorems. 

\begin{remark}
	
	In Theorems \ref{gaussianattractor} and \ref{stableattractor}, the
	limiting fields possess very irregular path properties. For example,
	if $\partial\mathcal{R}$ is strictly convex, then necessarily
	for any $p_{1},p_{2}$, $\mathcal{H}^{1}(\partial\mathcal{R}(p_{1})\cap\partial\mathcal{R}(p_{2}))=0$,
	which means that the fields appearing in Theorems \ref{gaussianattractor}
	and \ref{stableattractor} are white noises. 
	
\end{remark}

\begin{remark}By the previous remark, we have that the convergence
	in Theorems \ref{gaussianattractor} and \ref{stableattractor} cannot
	in general be strengthen to functional convergence. Moreover, since
	the convergence is stable and the limit is independent of the background
	driving L\'evy basis, we deduce that the convergence cannot take place
	in probability either.
	
\end{remark}

\begin{remark}The rates of convergence for $\mathscr{C}_{r}$ and
	$\mathcal{\mathscr{D}}_{r}$ can be seen as an $L^{\beta}$ norm of
	a certain parametrization of a disk. Indeed, let $D_{r}(p)$ be a
	disk of radius $r>0$ and center $p$, and put $g(s,\rho;\beta):=2\sqrt{1-s^{2}}\frac{(1+\beta)}{2}\rho$
	for $1\leq\beta\leq2$ . Then
	\[
	\left|D_{r}(p)\right|=\int_{-r}^{r}\int_{-1}^{1}\left|g(s,\rho;1)\right|\mathrm{d}s\mathrm{d}\rho.
	\]
	Moreover
	\[
	r^{1+1/\beta}v_{\beta}=\left(\int_{-r}^{r}\int_{-1}^{1}\left|g(s,\rho;\beta)\right|^{\beta}\mathrm{d}s\mathrm{d}\rho\right)^{1/\beta},
	\]
	so $r^{1+1/\beta}v_{\beta}$ can be thought as an $L^{\beta}$ norm
	of $g$.
	
\end{remark}

\subsection{Examples}

To clarify the results and the assumptions of Theorems \ref{gaussianattractor}-\ref{stableattractor},
in this part we present several examples.

\subsubsection*{Sets with positive reach}
Let $A$ be a Jordan domain such that $\inf_{ N(\partial A)}\delta_{\partial A}>0$. Then the integrability condition in Assumption \ref{assumptionambitset1} is satisfied. Indeed, we have in particular that $\delta_{A}$ and $\delta_{A^{*}}$ are bounded from below by, let's say $\varepsilon>0$. Thus, according to Theorems \ref{lemmadiffeo} and \ref{lemmadiffeo-1}
\[ 
\int_{\mathbb{S}^{1}} \sum_{q:(q,u)\in N(A)}(\delta_{A}(q,u)\wedge 1)\mathcal{H}^{1}(\mathrm{d}u)\leq\int_{\mathbb{S}^{1}} \sum_{q:(q,u)\in N(A)}\mathbf{1}_{\delta_A (q,u)>\varepsilon} \mathcal{H}^{1}(\mathrm{d}u)<\infty.
\]
Our claim follows by replacing $A$ by $A^{*}$ in the previous equation. Sets with the property $\inf_{ N(\partial A)}\delta_{\partial A}>0$ are known as {\it sets with positive reach}, see \cite{Federer59} for more details. It was shown in \cite{RatajLud17} that simple curves have positive reach if and only if are of class $C^{1,1}$, i.e. differentiable with Lipschitz derivative. Therefore, Jordan domains with boundary of positive reach satisfy Assumption \ref{assumptionambitset1}. 

\subsubsection*{Piecewise $C^{1,1}$ curves}
Let $A$ be a Jordan domain whose boundary is piecewise $C^{1,1}$.  This class of sets have indeed Lipschitz-regular boundary. However, as we saw above, $\partial A$ cannot have positive reach. Actually if $q_0\in\partial A$ is a corner then necessarily $\delta_{\partial A}(q_0,u)=0$ for any $u\in\mathbb{S}^{1}$. Nevertheless, Assumption \ref{assumptionambitset1} remains valid in this case. To see this, for simplicity assume that there is only one corner, say $q_0\in\partial A$. We can find $\rho>0$, such that the points in $D_{\rho}(q_0)\cap \partial A$ has null-curvature, i.e. two straight lines intersecting in $q_0$. Then outside of $D_{\rho}(q_0)\cap \partial A $, $\delta_{\partial A}$ is bounded (Corollary 8.9 in  \cite{RatajLud17}) from below and there are $u_1,u_2\in \mathbb{S}^{1}$ such that 
\[\mathcal{H}^0\{q\in D_{\rho}(q_0)\cap \partial A \setminus \{q_0\}:(q,u)\in N(A)\}=0,\,\,\,u\in \mathbb{S}^{1}\setminus\{u_1,u_2\}.
\]
The integrability condition in Assumption \ref{assumptionambitset1} follows from these observations. More generally, Jordan domains with piecewise $C^{1,1}$  boundary are within Assumption \ref{assumptionambitset1}.
\subsubsection*{Stable distributions}

Let $L$ be a homogeneous L\'evy basis with characteristic triplet $(\gamma,b,\nu)$.
Assume that the L\'evy seed $L'$ has a $\beta$-stable distribution
for $0<\beta\leq2$. Thus if $\beta=2$, we have that $L'$ is a Gaussian
r.v. with mean $\gamma$ and variance $b^{2}$ meaning that $L$ satisfies
the assumptions of Theorem \ref{gaussianattractor}. On the other
hand, if $\beta<2$ then
\[
\nu^{\pm}(x)=\beta^{-1}K_{\pm}x^{-\beta},\,\,\,x>0,
\]
meaning that, as $x\downarrow0$, $\nu^{\pm}(x)\sim\beta^{-1}K_{\pm}x^{-\beta}$.
Consequently, for $1<\beta<2$, $L$ is within the framework of Theorem
\ref{stableattractor}\textbf{ i.}, while for $\beta=1$, $L$ satisfies
the assumptions of Theorem \ref{stableattractor} \textbf{ii.} if
and only if $K_{+}=K_{-}$. Furthermore, for $\beta<1$, $L$ fulfills
the requirements of Theorem \ref{divergenceandvorticitytheorem}.
\subsubsection*{Generalized Hyperbolic distributions}

The family of Generalized hyperbolic distributions, originally introduced
in \cite{BN78}, constitutes a rich class of infinitely divisible
normal mean-variance mixture distributions.A r.v. $\xi$
is said to have generalized hyperbolic distribution with parameters $\lambda,\mu\in\mathbb{R}$,
$\delta>0$ and $0\leq\left|\theta\right|<\alpha$, and we write $\xi\sim GH(\lambda,\alpha,\theta,\delta,\mu)$,
if for $u\in\mathbb{R}$
\begin{align}
\mathcal{C}(u\ddagger\xi)&=\left[\mu+\theta\delta\frac{K_{\lambda+1}(\delta\sqrt{\alpha^{2}-\theta^{2}})}{\sqrt{\alpha^{2}-\theta^{2}}K_{\lambda}(\delta\sqrt{\alpha^{2}-\theta^{2}})}\right]iu\label{characteriscGH}\\
&+\int_{\mathbb{R}\backslash\{0\}}(e^{iux}-1-iux)v_{GH(\lambda,\alpha,\theta,\delta,\mu)}(x)\mathrm{d}x,\nonumber
\end{align}
where $K_{\zeta}$ denotes the modified Bessel function of second
kind with index $\zeta$ and
\[
v_{GH(\lambda,\alpha,\theta,\delta,\mu)}(x)=\frac{e^{\theta x}}{\left|x\right|}k_{\lambda,\alpha,\delta}(\left|x\right|) \mathbf{1}_{\left|x\right|>0},
\]
with $k_{\lambda,\alpha,\delta}$ satisfying 
\begin{equation}
k_{\lambda,\alpha,\delta}(x)=\frac{\delta}{\pi}x^{-1}+o(x^{-1}),\,\,\,\text{as }x\downarrow0.\label{GHk}
\end{equation}
For a closed form of $k_{\lambda,\alpha,\delta}$ the reader may consult
\cite{Raible00}. This means that the L\'evy measure of $\xi$ satisfies that $\int_{\mathbb{R}}(1\land\left|x\right|)v_{GH(\lambda,\alpha,\theta,\delta,\mu)}(x)\mathrm{d}x=+\infty$
and $\nu^{\pm}(x)\sim\frac{\delta}{\pi}x^{-1}$ as $x\downarrow0$.
Moreover, for any $y>0$
\[
\int_{y}^{1}xv_{GH(\lambda,\alpha,\theta,\delta,\mu)}(x)\mathrm{d}x+\int_{-1}^{-y}xv_{GH(\lambda,\alpha,\theta,\delta,\mu)}(x)\mathrm{d}x=\int_{y}^{1}[e^{\theta x}-e^{-\theta x}]k_{\lambda,\alpha,\delta}(x)\mathrm{d}x,
\]
so by the Monotone Convergence Theorem and (\ref{GHk}) we have that,
as $y\downarrow0$,
\[
\int_{y}^{1}[e^{\theta x}-e^{-\theta x}]k_{\lambda,\alpha,\delta}(x)\mathrm{d}x\rightarrow\int_{0}^{1}[e^{\theta x}-e^{-\theta x}]k_{\lambda,\alpha,\delta}(x)\mathrm{d}x<\infty.
\]
Thus, $L$ satisfies the assumptions of Theorem \ref{stableattractor}.

\subsubsection*{Non-negative L\'evy bases}

Any non-negative homogeneous L\'evy basis satisfies that $b=0$, $\int_{\mathbb{R}}(1\land\left|x\right|)\nu(\mathrm{d}x)<\infty$,\linebreak
$\nu[(-\infty,0)]=0$, and $\gamma_{0}=\gamma-\int_{0}^{1}x\nu(\mathrm{d}x)\geq0$.
For a proof of this fact we refer to \cite{BNPed12}.
In this case we have that $L$ can be written as
\[
L(A)=\gamma_{0}Leb(A)+\int_{A}\int_{0}^{\infty}xN(\mathrm{d}x,\mathrm{d}q),\,\,\,A\in\mathcal{B}_{b}(\mathbb{R}^{2}),
\]
where $N$ is a Poisson random measure with intensity $\nu(\mathrm{d}x)\mathrm{d}q$.
Hence, any non-negative L\'evy basis fulfills the condition of Theorem
\ref{divergenceandvorticitytheorem}.

\subsubsection*{Isotropic kernels}

Let $R_{\phi}$ be the rotation matrix on $\mathbb{R}^{2}$ and $f$
a continuous function. Put
\[
F_{\phi,f}(q)=f(\left\Vert q\right\Vert )R_{\phi}q,\,\,\,q\in\mathbb{R}^{2}.
\]
In the next section we will show that when the ambit set is isotropic,
meaning that it can be written as 
\[
\mathcal{R}=\{q\in\mathbb{R}^{2}:h(\left\Vert q\right\Vert )\in A\},
\]
for some measurable function $h$ and $A\subset\mathbb{R}$, then
the ambit field of the form of (\ref{eq:ambitfield}) induced by $F_{\phi,f}$
and $\mathcal{R}$, has isotropic increments (see next section for
a precise definition). Now, if we let $h(x)=x$ and $A=[a,b]$, with
$0\leq a<b$, then $\mathcal{R}$ is an annulus, meaning that Assumption
\ref{assumptionambitset1} is satisfied. Moreover, if $f(a)=f(b)=0$,
then $\left.F\right|_{-\mathcal{\partial R}}\equiv0$,
or in other words, the conclusion of Theorem \ref{divergenceandvorticitytheorem}
holds for any L\'evy basis, whenever $f$ is continuously differentiable
on $[a,b]$. On the other hand, if
we put 
\[
f(x)=x^{-2},\,\,\,x>0,
\]
then $F_{\phi,f}$ is continuously differentiable on $\mathcal{R}$
if and only if $a>0$.

\section{A class of incompressible and rotational ambit fields }

The main goal of this section is to build a class of ambit fields
that have homogeneous and isotropic increments as well as being rotational
and having the property of incompressibility. Let us introduce the
formal definition of these concepts.

\begin{definition}An $\mathbb{R}^{2}$-valued random field $(Y(p))_{p\in\mathbb{R}^{2}}$
	is said to have {\it homogeneous} and {\it isotropic} increments
	if respectively, the following two conditions hold
	\begin{itemize}
		\item For any $p_{0}\in\mathbb{R}^{2}$ the field $(Y(p+p_{0})-Y(p))_{p\in\mathbb{R}^{2}}$
		is stationary;
		\item For any $p_{0}\in\mathbb{R}^{2}$ and $\theta\in\left[0,2\pi\right)$
		we have that 
		\[
		\left\{ R_{\theta}^{-1}[Y(R_{\theta}(p+p_{0}))-Y(R_{\theta}p)]\right\} _{p\in\mathbb{R}^{2}}\overset{d}{=}\left\{ [Y(p+p_{0})-Y(p)]\right\} _{p\in\mathbb{R}^{2}}.
		\]
	\end{itemize}
	Furthermore, we will say that it is {\it incompressible} if for any
	$p\in\mathbb{R}^{2}$ 
	\[
	\lim_{r\downarrow0}\frac{1}{\pi r^{2}}\ointop_{r\mathbb{S}^1(p)}Y\cdot n\mathrm{d}s\overset{\mathbb{P}}{\rightarrow}0,
	\]
	and it will be called {\it rotational} if the following limit exists and it is not constantly zero
	\[
	\mathbb{P}-\lim_{r\downarrow0}\frac{1}{\pi r^{2}}\ointop_{r\mathbb{S}^1(p)}Y\cdot n_{D}^{\perp}\mathrm{d}s.
	\]
	
\end{definition}
Now, given a real-valued continuous function $f$ and $\phi\in\left[0,2\pi\right)$,
let 
\[
F_{\phi,f}(q):=R_{\phi}qf(\left\Vert q\right\Vert )
\]
and consider the class of ambit fields than can be written as
\begin{equation}
Y_{\phi,f}(p):=\int_{\mathcal{R}+p}F_{\phi,f}(p-q)V(q)L\left(\mathrm{d}q\right),\,\,\,p\in\mathbb{R}^{2};\phi\in\left[0,2\pi\right),\label{ambitmodel}
\end{equation}
where $L$ is a homogeneous L\'evy basis with characteristic triplet
$(\gamma,b,\nu)$ and $V$ a predictable real-valued field. Moreover,
we let
\[
\mathcal{R}=\{q\in\mathbb{R}^{2}:h(\left\Vert q\right\Vert )\in A\},
\]
for some measurable function $h$ and $A\subset\mathbb{R}$, in such
a way that $\mathcal{R}$ is compact. As the following result shows,
this family of ambit fields is well defined and has isotropic and
homogeneous increments.

\begin{proposition}\label{propisotropyambit}Suppose that $V$ is
	predictable (see Appendix A) and locally bounded. Then $Y_{\phi,f}$
	as in (\ref{ambitmodel}) is well defined. If in addition we have
	that $V$ is bounded in $\mathcal{L}^{2}(\Omega,\mathcal{F},\mathbb{P})$,
	then $Y_{\phi,f}$ is continuous in probability. Finally, if $V$
	is independent of $L$ and almost surely
	\begin{equation}
	\int_{\mathcal{R}}\int_{\left|xV(q)\right|>1}\left|xV(q)\right|\nu(\mathrm{d}x)\mathrm{d}q<\infty,\label{bigjumpsVL}
	\end{equation}
	then $Y_{\phi,f}$ is stationary with finite first moment and possesses
	isotropic and homogeneous increments.
	
\end{proposition}

\begin{proof}From Proposition \ref{convstchintambittype} and its subsequent remark in Appendix A,
	we have that $Y_{\phi,f}$ is well defined if and only if a.s.
	\[
	\int_{\mathcal{R}}\Phi_{L}^{0}[\left\Vert F_{\phi,f}(-q)\right\Vert \left|V(p+q)\right|]\mathrm{d}q<\infty.
	\]
	Since $V$ is locally bounded, we have that almost surely $\left\Vert F_{\phi,f}(q)\right\Vert \left|V(p-q)\right|\leq M_{p}$
	for any $q\in\mathcal{R}$, for some r.v. $M_{p}$ only depending
	on $p$. Hence, by Lemma 2.1.5 in \cite{Rosinski07}, we have 
	\[
	\int_{\mathcal{R}}\Phi_{L}^{0}[\left\Vert F_{\phi,f}(q)\right\Vert \left|V(p-q)\right|]\mathrm{d}q\leq 2\Phi_{L}^{0}(M_{p})Leb(\mathcal{R})<\infty,
	\]
	showing this the well definiteness of $Y_{\phi,f}$. Now, if $V$
	is bounded on $\mathcal{L}^{2}(\Omega,\mathcal{F},\mathbb{P})$, the stochastic continuity of $Y_{\phi,f}$ can be shown in a similar way as in the proof of Lemma \ref{lemmadecomposition} below.
	
	In general, $Y_{\phi,f}$ is stationary whenever $V$ is stationary
	and independent of $L$, so it has homogeneous increments in this
	situation. On the other hand if $V$ is independent of $L$, and for $p_{0},p_{1},\ldots,p_{n}\in\mathbb{R}^{n}$,
	we let $Y_{\phi,f}(\overrightarrow{p},p_{0}):=\{Y(p_{i}+p_{0})-Y(p_{i})\}_{i=1}^{n}$, then
	\[
	\mathbb{E}(exp\{i\left\langle z,Y_{\phi,f}(\overrightarrow{p},p_{0})\right\rangle \})=\mathbb{E}exp\left(\left\{ \int_{\mathbb{R}^{2}}\mathcal{C}[\left\langle z,F_{\phi,f}(\overrightarrow{p},p_{0},q)\right\rangle V(q)]\mathrm{d}q\right\} \right),
	\]
	where 
	\[
	F_{\phi,f}(\overrightarrow{p},p_{0},q)=\{F_{\phi,f}(p_{i}+p_{0}-q)\mathbf{1}_{\mathcal{R}+p_{i}+p_{0}}(q)-F_{\phi,f}(p_{i}-q)\mathbf{1}_{\mathcal{R}+p_{i}}(q)\}_{i=1}^{n}.
	\]
	Hence, if $\theta\in\left[0,2\pi\right)$ then 
	\[
	\int_{\mathbb{R}^{2}}\mathcal{C}[\left\langle z,R_{\theta}^{-1}F_{\phi,f}(R_{\theta}\overrightarrow{p},p_{0},q)\right\rangle V(q)]\mathrm{d}q=\int_{\mathbb{R}^{2}}\mathcal{C}[\left\langle z,R_{\theta}^{-1}F_{\phi,f}(R_{\theta}\overrightarrow{p},p_{0},R_{\theta}q)\right\rangle V(R_{\theta}q)]\mathrm{d}q,
	\]
	where we have done the change of variable $q=R_{\theta}q'$. But in
	view of
	\[
	R_{\theta}^{-1}F_{\phi,f}(R_{\theta}\overrightarrow{p},p_{0},R_{\theta}q)=F_{\phi,f}(\overrightarrow{p},p_{0},q),
	\]
	we get that
	\begin{align*}
	\mathbb{E}(exp\{i\left\langle z,R_{\theta}^{-1}Y_{\phi,f}(R_{\theta}\overrightarrow{p},p_{0})\right\rangle \})&=\mathbb{E}(exp\{\int_{\mathbb{R}^{2}}\mathcal{C}[\left\langle z,F_{\phi,f}(\overrightarrow{p},p_{0},q)\right\rangle V(R_{\theta}q)]\mathrm{d}q\})\\
	&=\mathbb{E}(exp\{i\left\langle z,Y_{\phi,f}(\overrightarrow{p},p_{0})\right\rangle \}),
	\end{align*}
	which shows that $Y_{\phi,f}$ has isotropic increments. Finally,
	assume that (\ref{bigjumpsVL}) holds. Since $V$ is independent of
	$L$, then, conditioned on $V$, $Y_{\phi,f}(0)$ is an ID random variable. Suppose for a moment that almost surely $\int_{\left\Vert y\right\Vert >1}\left\Vert y\right\Vert \nu_{Y_{\phi,f}}(\mathrm{d}y)<\infty$.
	Then $\mathbb{E}(\left\Vert Y_{\phi,f}(0)\right\Vert \left|V\right.)<\infty$
	a.s. and (see Example 25.12 in \cite{Sato99})
	\[  
	\mathbb{E}(Y_{\phi,f}(0)\left|V\right)=\gamma\int_{\mathcal{R}}F_{\phi,f}(-q)V(q)\mathrm{d}q+\int_{\mathcal{R}}\int_{\left|x\right|>1}xF_{\phi,f}(-q)V(q)\nu(\mathrm{d}x)\mathrm{d}q\mathrm{d}s,\]
	which would imply immediately the existence of $\mathbb{E}(Y_{\phi,f}(0))$.
	Therefore, it only remains to show that almost surely $\int_{\left\Vert y\right\Vert >1}\left\Vert y\right\Vert \nu_{Y_{\phi,f}}(\mathrm{d}y)<\infty$.
	Observe that (\ref{bigjumpsVL}) together with the stationarity and
	the local boundedness of $V$, implies that there is $\Omega_{0}\in\mathcal{F}$
	with $\mathbb{P}(\Omega_{0})=1$, such that for any $\omega\in\Omega_{0}$,
	$\int_{\left|xV(q)(\omega)\right|>1}\left|x\right|\nu(\mathrm{d}x)<\infty$
	for almost all $q\in\mathcal{R}$, and $0<\sup_{\mathcal{R}}\left|V(q)\right|(\omega)\leq M(\omega)<\infty$
	for some $M(\omega)>0$. Thus, for $\omega\in\Omega_{0}$
	\[int_{\left|y\right|>1}\left|y\right|\nu_{Y_{\phi,f}}(\mathrm{d}y)(\omega)  \leq M(\omega)M_{F}Leb(\mathcal{R})\int_{\left|xM(\omega)M_{F}\right|>1}\left|x\right|\nu(\mathrm{d}x)<\infty,  \]
	where $0<M_{F}=\sup_{\mathcal{R}}\left\Vert F_{\phi,f}(-q)\right\Vert <\infty$. This concludes the proof. \end{proof}

\begin{remark}\label{remarkgenambitfieldsproperties}Observe that,
	besides the isotropy of the increments, all the stated conclusions
	of the previous proposition are valid for the class of ambit fields
	of the form of 
	\begin{equation}
	Y(p)=\int_{\mathcal{R}+p}F(p-q)V(q)L(\mathrm{d}q),\label{ambitfieldform}
	\end{equation}
	where $\mathcal{R}$ is compact and $F$ continuously differentiable.
	
\end{remark}

Next, we proceed to characterize incompressible fields of the form
of (\ref{ambitmodel}) under the framework of Assumption \ref{assumptionambitset1}.
This is done by extending the results of Theorem \ref{divergenceandvorticitytheorem}
to the context of ambit fields. The proof of the following result
will be presented in the Section \ref{sec:Proofs}.

\begin{theorem}\label{thmdivvortambit}Let $V$ satisfy the whole asumptions of the previous propositions and independent of $L$. If $Y$ is as in (\ref{ambitfieldform}), then Theorem \ref{divergenceandvorticitytheorem}
	remains valid when we replace $L(\mathrm{d}q)$ by $V(q)L(\mathrm{d}q)$ and $\tilde{L}$ by $L$.
	
\end{theorem}

An application of this result gives us that:

\begin{proposition} Let $V$ and $L$ be as in
	the previous propositionLet $Y_{\phi,f}$ be as in (\ref{ambitmodel}),
	$\mathcal{R}$ fulfilling Assumption \ref{assumptionambitset1} and
	$F_{\phi,f}$ being continuously differentiable in $\mathcal{R}$.
	Then we have the following:
	\begin{enumerate}
		\item $Y_{\phi,f}$ is incompressible if either $\text{\ensuremath{\phi}=}\frac{\pi}{2},\frac{3\pi}{2}$
		and $f$ arbitrary, or $\phi\neq\frac{\pi}{2},\frac{3\pi}{2}$ and
		for some constant $K\in\mathbb{R}$
		\[
		f(x)=Kx^{-2},\,\,\,x>0.
		\]
		\item $Y_{\phi,f}$ is irrotational if either $\text{\ensuremath{\phi}=}0,\pi$
		and $f$ arbitrary, or $\phi\neq0,\pi$ and for some constant $K\in\mathbb{R}$
		\[
		f(x)=Kx^{2},\,\,\,x>0.
		\]
	\end{enumerate}
\end{proposition}

\begin{proof}From Theorem \ref{thmdivvortambit}, we have that
	
	\begin{align*}
	\frac{1}{\pi r^{2}}\ointop_{r\mathbb{S}^1(p)}Y_{\phi,f}\cdot n\mathrm{d}s\overset{\mathbb{P}}{\rightarrow}\int_{\mathcal{R}+p}\nabla\cdot F_{\phi,f}(p-q)V(q)L(\mathrm{d}q);\\
	\frac{1}{\pi r^{2}}\ointop_{r\mathbb{S}^1(p)}Y\cdot n_{D}^{\perp}\mathrm{d}s\overset{\mathbb{P}}{\rightarrow}\int_{\mathcal{R}+p}\nabla^{\perp}\cdot F_{\phi,f}(p-q)V(q)L(\mathrm{d}q).
	\end{align*}
	Therefore, $Y_{\phi,f}$ is incompressible if and only if the divergence
	of $F_{\phi,f}$ vanishes. In an similar way we see that $Y_{\phi,f}$
	is irrotational if and only if the
	curl of $F_{\phi,f}$ vanishes. The conclusions of this proposition
	follow then from the previous observations and the fact that
	\begin{align*}
	\nabla\cdot F_{\phi,f}(q) & =2\cos(\phi)f(\left\Vert q\right\Vert )+\cos(\phi)f'(\left\Vert q\right\Vert )\left\Vert q\right\Vert ;\\
	\nabla^{\perp}\cdot F_{\phi,f}(q) & =2\sin(\phi)f(\left\Vert q\right\Vert )-\sin(\phi)f'(\left\Vert q\right\Vert )\left\Vert q\right\Vert .
	\end{align*}
\end{proof} 

\begin{remark}The previous proposition together with Proposition
	\ref{propisotropyambit}, show that it is always possible to construct
	an ambit field that has homogeneous and isotropic increments as well
	as being incompressible and rotational. Then a natural question arises:
	Is there another class of kernels different from $F_{\phi,f}$ having
	these properties?
	
\end{remark} 

\section{Proofs \label{sec:Proofs}}

In this part we present the proofs of our main results. We will proceed
as follow: First, we establish some preliminary results on the existence
and the representation of the functionals $\mathscr{D}_{r}(\cdot;X)$
and $\mathscr{C}_{r}(\cdot;X)$ when $X$ is of the form (\ref{IDXdef}).
We use such a representation to formalize the decomposition (\ref{decompositiondiv})
discussed in Subsection \ref{subsec:Some-intuitive-description}.
We apply this to identify the part of $\mathscr{D}_{r}(\cdot;X)$
and $\mathscr{C}_{r}(\cdot;X)$ that dominates the asymptotics. We
then invoke Assumption \ref{assumptionambitset1} to fully describe
the asymptotic rates of such a dominating part. The conclusions of
Theorems \ref{gaussianattractor}-\ref{stableattractor} will follow
from this. The proof of Theorem \ref{thmdivvortambit} will basically
be a consequence of Theorem \ref{divergenceandvorticitytheorem} and
Proposition \ref{convstchintambittype} in Appendix A. Furthermore, since our reasoning is independent to whether we use $n$ or $n^{\perp}$, we will only present the proof of those results concerned to $\mathscr{D}_{r}(\cdot;X)$.

\subsection{Preliminary results and remarks}

In what follows $L$ will denote a real-valued homogeneous L\'evy basis
on $\mathbb{R}^{d}$ with characteristic triplet $(\gamma,b,\nu)$
and $X$ as in (\ref{IDXdef}). Let us start by a very simple observation:
$\mathscr{D}_{r}(\cdot;X)$ vanishes
when $L$ is deterministic. Indeed, suppose that $L(A)=\gamma Leb(A)$
for some $\gamma\in\mathbb{R}$. Then for any $p\in\mathbb{R}^{2},$
$X(p)\equiv X(0)\equiv\gamma\int_{\mathcal{R}}F(-q)\mathrm{d}q,$
meaning this that 
\[ \mathscr{D}_{r}(p;X)=  r\int_{0}^{2\pi}X(0)\cdot u(\theta)\mathrm{d}\theta=0, \]
as claimed. In fact, a deterministic homogeneous field is necessarily
constant. Hence, we can conclude for this and the L\'evy-It\^o decomposition
for L\'evy bases (see \cite{Ped03}) that for any $p\in\mathbb{R}^{2}$
\begin{equation}
\mathscr{D}_{r}(p;X)=\mathscr{D}_{r}(p;\tilde{X}),\label{DandCnodrift}
\end{equation}
where 
\[
\tilde{X}(p)=\int_{\mathcal{R}+p}F(p-q)\tilde{L}(\mathrm{d}q),
\]
and $\tilde{L}$ as in Theorem \ref{divergenceandvorticitytheorem}.

Now, by using  Lemma \ref{Fubinilemma} in Appendix B, we formalize (\ref{decompositiondiv}). 

\begin{lemma}\label{lemmadecomposition}Let $\mathcal{R}\subset\mathbb{R}^{2}$
	be compact and $F:\mathbb{R}^{2}\rightarrow\mathbb{R}^{2}$ a measurable
	function that is continuously differentiable on $-\mathcal{R}$. Then
	the field
	\begin{equation}
	X(p):=\int_{\mathcal{R}+p}F(p-q)L(\mathrm{d}q),\,\,\,p\in\mathbb{R}^{2},\label{IDXdef-1}
	\end{equation}
	is well defined and continuous in probability. Moreover, the functionals
	$\mathscr{C}_{r}(\cdot;X)$ and $\mathcal{\mathscr{D}}_{r}(\cdot;X)$
	given in (\ref{generalizedcirculation}) and (\ref{eq:generalizedflux}),
	are well defined and they can be decomposed as
	\begin{align}
	\mathscr{C}_{r}(p;X)= & \mathring{\mathscr{C}}_{r}(p;X)+\partial\mathscr{C}_{r}(p;X);\label{curldeco}\\
	\mathcal{\mathscr{D}}_{r}(p;X)= & \mathring{\mathscr{D}}_{r}(p;X)+\partial\mathscr{D}_{r}(p;X),\label{Divdecom}
	\end{align}
	where 
	\begin{align*}
	\mathring{\mathscr{D}}_{r}(p;X) & :=\int_{\mathcal{R}(p)\backslash(\partial\mathcal{R}(p))_{\oplus r}}\int_{D_{r}(p-q)}\nabla\cdot F(y)\mathrm{d}yL\left(\mathrm{d}q\right);\\
	\partial\mathscr{D}_{r}(p;X) & :=\int_{(\partial\mathcal{R}(p))_{\oplus r}}\ointop_{r\mathbb{S}^1(q)}F(p-\cdot)\mathbf{1}_{\mathcal{R}(p)}\cdot n\mathrm{d}sL\left(\mathrm{d}q\right),
	\end{align*}
	and 
	\begin{align*}
	\mathring{\mathscr{C}}_{r}(p;X):= & \int_{\mathcal{R}(p)\backslash(\partial\mathcal{R}(p))_{\oplus r}}\int_{D_{r}(p-q)}\nabla^{\perp}\cdot F(y)\mathrm{d}yL\left(\mathrm{d}q\right);\\
	\partial\mathscr{D}_{r}(p;X):= & \int_{(\partial\mathcal{R}(p))_{\oplus r}}\ointop_{r\mathbb{S}^1(q)}F(p-u)\mathbf{1}_{\mathcal{R}(p)}(u)\cdot\mathrm{d}uL\left(\mathrm{d}q\right).
	\end{align*}
	
\end{lemma}

\begin{proof}The proof will be divided in three steps. First we prove
	the existence and the stochastic continuity of $X$. Second, we show
	that $\mathscr{D}_{r}(p;X)$ is well defined by verifying that the conditions of Lemma \ref{Fubinilemma}  in
	Appendix B, are satisfied. The decomposition (\ref{Divdecom})
	will follow from this.
	
	\textbf{Existence and continuity in probability:} From Lemma 2.1.5 in \cite{Rosinski07}, we have that 
	\[
	\int_{\mathcal{R}}\Phi_{L}^{0}\left(\left\Vert F(-q)\right\Vert \right)\mathrm{d}q\leq2Leb(\mathcal{R})\Phi_{L}^{0}(c_{F})<\infty,
	\]
	where $c_{F}=\sup_{\mathcal{R}}\left\Vert F(-q)\right\Vert <\infty$.
	This shows that $X$ is well defined. Now, pick $p_{n}\rightarrow0$.
	Using the same lemma, we have that
	\begin{align*}
	\Phi_{L}^{0}\left(\left\Vert F(p_{n}-q)\mathbf{1}_{\mathcal{R}(p_{n})}-F(-q)\mathbf{1}_{\mathcal{R}}\right\Vert \right) \leq2c_{F,\Phi}\mathbf{1}{}_{\mathcal{R}\cup\mathcal{R}(p_{n})},
	\end{align*}
	where $c_{F,\Phi}=\max\{\Phi_{L}^{0}(c_{F}),\Phi_{L}^{0}(2c_{F})\}$.
	Thus, by the Dominated Convergence Theorem, we deduce that, as $n\rightarrow\infty$,
	\[
	\int_{\mathbb{R}^{2}}\Phi_{L}^{0}\left(\left\Vert F(p_{n}-q)\mathbf{1}_{\mathcal{R}(p_{n})}-F(-q)\mathbf{1}_{\mathcal{R}}\right\Vert \right)\mathrm{d}q\rightarrow0,
	\]
	which according to Theorem \ref{DCTstcint} in Appendix , shows that $X$ is continuous in probability. 
	
	\textbf{Application of Stochastic Fubini Theorem}: We now verify that 1.-3. in Lemma \ref{Fubinilemma} in
	Appendix B are satisfied
	for the function $f_{\mathscr{D}}(p,q):=\mathbf{1}_{\mathcal{R}+p} F(p-q)\cdot n(p) $ where we consider $K=[0,2\pi]$, $\varphi(\theta)=p+ru(\theta)$
	with $u(\theta)=(\cos\theta,\sin\theta)'$ and $\left|D\varphi(\theta)\right|=r$. By the continuity of $F$, one can deduce that \linebreak $\left|\gamma\right|r\int_{0}^{2\pi}\int_{\mathcal{R}}\left|F(-q)\cdot u(\theta)\right|\mathrm{d}q\mathrm{d}\theta<\infty$
	and $b^{2}r\int_{0}^{2\pi}\int_{\mathcal{R}}\left|F(-q)\cdot u(\theta)\right|\mathrm{d}q\mathrm{d}\theta<\infty$. Similarly, we get that
	\[
	\int_{0}^{2\pi}\int_{\mathcal{R}}\int_{\left|x\right|\leq1}\left|xF(-q)\cdot u(\theta)\right|^{2}\nu(\mathrm{d}x)\mathrm{d}q\mathrm{d}\theta<\infty.
	\]
	Thus, 1. and 2. hold. Hence, it only remains to verify 3. Before doing this, let us first note that for a fixed $r>0$, $\mathbf{1}_{\mathcal{R}+p+ru(\theta)}(q)=0$ if $d_{\mathcal{R}(p)}(q)>r$. In view of the previous observation, we get that $\chi(q):=\int_Kf_{\mathscr{D}}(\varphi(\theta),q)\left|D\varphi(\theta)\right|\mathrm{d}\theta=0$, whenever $q\in\mathcal{I}_{r}(p)^{c}$, where 
	\begin{equation}
	\mathcal{I}_{r}(p):=\mathcal{R}(p)\cup(\partial\mathcal{R}(p))_{\oplus r}.\label{regioninteIr}
	\end{equation}
	Consequently
	\[
	\int_{\mathbb{R}}\int_{\left|x\right|>1}(\left|x\chi(q)\right|\land1)\nu(\mathrm{d}x)\mathrm{d}q\leq\int_{\mathcal{I}_{r}(p)}\int_{\left|x\right|>1}\nu(\mathrm{d}x)\mathrm{d}q<\infty,
	\]
	i.e. 3. is satisfied.
	
	\textbf{Decomposition: }We now proceed to show that  (\ref{Divdecom}) is valid. By the previous part and Lemma \ref{Fubinilemma}  in
	Appendix B, we can conclude that the functional $\mathscr{D}_{r}(p;X)$ is well defined and almost surely
	\[
	\mathscr{D}_{r}(p;X)=r\int_{\mathcal{I}_{r}(p)}\int_{0}^{2\pi}f_{\mathscr{D}}(p+ru(\theta),q)\mathrm{d}\theta L(\mathrm{d}q),\,\,\,r>0,p\in\mathbb{R}^{2},
	\]
	with $\mathcal{I}_{r}(p)$ as in (\ref{regioninteIr}). On the other
	hand, by the definition of $(\partial\mathcal{R}(p))_{\oplus r}$,
	one get that 
	\begin{align*}
	\ointop_{r\mathbb{S}^1(p-q)}F\mathbf{1}_{-\mathcal{R}}\cdot n\mathrm{d}s & =r\int_{0}^{2\pi}f_{\mathscr{D}}(p+ru(\theta),q)\mathrm{d}\theta\\
	& =\begin{cases}
	r\int_{\{\theta:q-ru(\theta)\in\mathcal{R}(p)\}}F(p-(q-ru(\theta)))\cdot n(\theta)\mathrm{d}\theta, & q\in(\partial\mathcal{R}(p))_{\oplus r};\\
	r\int_{0}^{2\pi}F(p-q+ru(\theta))\cdot n(\theta)\mathrm{d}\theta, & otherwise,
	\end{cases}
	\end{align*}
	and by the usual Stokes Theorem, we have
	\[
	\ointop_{r\mathbb{S}^1(p-q)}F\cdot n\mathrm{d}s=\int_{D_{r}(p-q)}\nabla\cdot F(y)\mathrm{d}y,\,\,\,q\in\mathcal{R}(p)\backslash(\partial\mathcal{R}(p))_{\oplus r}.
	\]
	The decomposition appearing in (\ref{Divdecom}) follows from this.\end{proof}

\subsection{Identification of the rates}

In this part, by using the decomposition obtained in Lemma \ref{lemmadecomposition},
we identify the dominating part of $\mathscr{C}_{r}(p;X)$ and $\mathscr{D}_{r}(p;X)$. 

\subsubsection*{Asymptotics for $\mathring{\mathscr{C}}_{r}(p;X)$ and $\mathring{\mathscr{D}}_{r}(p;X)$}

The following lemma reveals that, as in the temporal case, the asymptotic
rate for $\mathring{\mathscr{C}}_{r}(p;X)$ and $\mathring{\mathscr{D}}_{r}(p;X)$
is the ``classical'' one, i.e. $\pi r^{2}$.

\begin{proposition}\label{asymptoticsint}Let $X,$ $\mathring{\mathscr{C}}_{r}(p;X)$
	and $\mathring{\mathscr{D}}_{r}(p;X)$ be as in Lemma \ref{lemmadecomposition}.
	Then for every $p\in\mathbb{R}^{2}$ we have that as $r\downarrow0$
	\[
	\frac{1}{\pi r^{2}}\mathring{\mathscr{D}_{r}}(p;X)\overset{\mathbb{P}}{\rightarrow}\tilde{\sigma}(p);\,\,\,\frac{1}{\pi r^{2}}\mathring{\mathscr{D}_{r}}(p;X)\overset{\mathbb{P}}{\rightarrow}\tilde{\omega}(p).
	\]
	where
	\[
	\tilde{\sigma}(p):=\int_{\mathcal{R}+p}\mathrm{\nabla}\cdot F(p-q)L(\mathrm{d}q);\,\,\,\tilde{\omega}(p):=\int_{\mathcal{R}+p}\nabla^{\perp}\cdot F(p-q)L(\mathrm{d}q)
	\]
\end{proposition}

\begin{proof}Let us first note that since $F$ is continuously differentiable
	on $-\mathcal{R}$, we can argue as in the first step of the proof
	of Lemma \ref{lemmadecomposition} in order to deduce that
	\[
	\int_{\mathcal{R}}\left[\Phi_{L}^{0}(\left|\nabla^{\perp}\cdot F(-q)\right|)+\Phi_{L}^{0}(\left|\mathrm{\nabla}\cdot F(-q)\right|)\right]\mathrm{d}q<\infty,
	\]
	which shows that the fields $\omega$ and $\sigma$ are well defined. In order to finish the proof, thanks to Theorem \ref{DCTstcint} in
	Appendix A, the previous lemma and stationarity, we only
	need to verify that as $r\downarrow0$
	\begin{equation}
	\int_{\mathcal{R}\backslash(\partial\mathcal{R}(p))_{\oplus r}}\Phi_{L}^{0}(h_{r,\mathscr{D}}(q))\mathrm{d}q\rightarrow0,\label{condconvergenceintC}
	\end{equation}
	where for $q\in\mathcal{R}\backslash(\partial\mathcal{R}(p))_{\oplus r}$
	\[ 	h_{r,\mathscr{D}}(q):=  \frac{1}{\pi r^{2}}\int_{D_{r}(-q)}\left\{ \nabla\cdot F(y)-\nabla\cdot F(-q)\right\} \mathrm{d}y. \]
	In the present framework, the mapping $y\mapsto\nabla\cdot F(y)$
	is continuous on $-\mathcal{R}$. Thus, since $\mathcal{R}$ is
	compact, the Tietze Extension Theorem allow us to extend continuously
	$h_{r,\mathscr{D}}$ on $\mathbb{R}^{2}$, in such a way
	\[
	\left|h_{r,\mathscr{D}}(q)\right|\leq2\sup_{\mathbb{R}^{2}}\left|\nabla\cdot F(q)\right|<\infty,\,\,\,q\in\mathbb{R}^{2}.
	\]
	Therefore, by Lemma 2.1.5 in \cite{Rosinski07}, we obtain that 
	\[
	\left|\Phi_{L}^{0}(h_{r,\mathscr{D}}(q))\right|\mathbf{1}_{\mathcal{R}\backslash(\partial\mathcal{R}(p))_{\oplus r}}(q)\leq2\Phi_{L}^{0}(c_{1}),
	\]
	where $c_{1}=2\sup_{\mathbb{R}^{2}}\left|\nabla\cdot F(q)\right|$.
	Equation (\ref{condconvergenceintC}) then follows as an application
	of the Lebesgue Differentiation Theorem and the Dominated Convergence
	Theorem.\end{proof}

Proposition \ref{asymptoticsint} formalizes the intuition discussed
in Subsection \ref{subsec:Some-intuitive-description}, i.e.
we have shown that if $X$ is of the form of (\ref{IDXdef-1}), then
the limit behavior of $\mathscr{D}_{r}(p;X)$
is completely determined by $\partial\mathscr{D}_{r}(p;X)$,
which is in turn depending on $L$, $F$ and the geometry of
$\partial\mathcal{R}$. 

\subsubsection*{The functionals $\partial\mathscr{C}_{r}(p;X)$ and $\partial\mathscr{D}_{r}(p;X)$
	are negligible in the compound Poisson case}

Not surprisingly, as the following lemma shows, when the L\'evy seed
of $L$ has a compound Poisson distribution we get that $\partial\mathscr{C}_{r}(p;X)$
and $\partial\mathscr{D}_{r}(p;X)$ are negligible.

\begin{proposition}\label{propboundarycomppoisso}Let $X,$ $\partial\mathscr{C}_{r}(p;X)$
	and $\partial\mathscr{D}_{r}(p;X)$ be as in Lemma \ref{lemmadecomposition}
	and assume that $L$', the L\'evy seed of $L$, satisfies
	\[
	\mathcal{C}(z\ddagger L')=\int_{\mathbb{R}}(e^{izx}-1)\nu(\mathrm{d}x),\,\,\,z\in\mathbb{R},
	\]
	where $\nu(\mathbb{R})<\infty$. Let $(a_{r})_{r>0}$ be any collection
	of real number such that $a_{r}\rightarrow0$ as $r\downarrow0$.
	Then for all $p\in\mathbb{R}^{2}$ we have that
	\[
	a_{r}^{-1}\partial\mathscr{C}_{r}(p;X)\overset{\mathbb{P}}{\rightarrow}0;\,\,\,a_{r}^{-1}\partial\mathscr{D}_{r}(p;X)\overset{\mathbb{P}}{\rightarrow}0.
	\]
\end{proposition}

\begin{proof}By Lemma \ref{lemmadecomposition} and our assumption, as $r\downarrow0$
	\[ 		\left|\mathcal{C}(z\ddagger a_{r}^{-1}\partial\mathscr{C}_{r}(p;X))\right|\leq2c_{L'}Leb((\partial\mathcal{R}(p))_{\oplus r})\rightarrow0,\]
	where $c_{L'}=\sup_{\mathbb{R}}\left|\mathcal{C}(z\ddagger L')\right|<\infty$.\end{proof}

\begin{remark}Observe that at this point none of the results of this
	section have used Assumption \ref{assumptionambitset1}. We can therefore
	infer that Theorem \ref{divergenceandvorticitytheorem} holds for
	any field of the form of (\ref{IDXdef-1}) for which $F$ is continuously
	differentiable, $\mathcal{R}$ is compact and the L\'evy seed of $L$
	is distributed according to a compound Poisson distribution. Assumption
	\ref{assumptionambitset1} is crucial when this is not the case.
	
\end{remark}

\subsubsection*{A fundamental decomposition of $\partial\mathscr{C}_{r}(p;X)$ and
	$\partial\mathscr{D}_{r}(p;X)$}

Under the framework of Assumption \ref{assumptionambitset1}, we have
that $\partial\mathcal{R}=\bigcup_{i=1}^{n}\partial\mathcal{R}_{i}$
where all the $\partial\mathcal{R}_{i}'s$ are disjoint and closed.
This implies, in particular, that for $r$ small enough the parallel
sets of $\partial\mathcal{R}$ satisfies that
\[
(\partial\mathcal{R})_{\oplus r}=\bigcup_{i=1}^{n}(\partial\mathcal{R}_{i})_{\oplus r},
\]
with $(\partial\mathcal{R}_{i})_{\oplus r}$ disjoint. Consequently,
by Lemma \ref{lemmadecomposition}, almost surely
\begin{align}
\partial\mathscr{D}_{r}(p;X)= & \int_{(\partial\mathcal{R}_{1}+p)_{\oplus r}}\ointop_{r\mathbb{S}^1(q)}F(p-\cdot)\mathbf{1}_{\mathcal{R}_{1}(p)}\cdot n\mathrm{d}sL\left(\mathrm{d}q\right)\label{decompositionTauD}\\
& \sum_{j=1}^{m}\int_{(\partial\mathcal{R}_{j}+p)_{\oplus r}}\ointop_{r\mathbb{S}^1(q)}F(p-\cdot)\mathbf{1}_{\mathcal{R}_{j}^{*}(p)}\cdot n\mathrm{d}s\mathrm{d}L\left(\mathrm{d}q\right).\nonumber 
\end{align}
where the summands are independent. An analogous decomposition holds for $\partial\mathscr{D}_{r}(p;X)$. $\partial\mathscr{C}_{r}(p;X)$ and $\partial\mathscr{D}_{r}(p;X)$
consist of independent sums of functionals of the type of 
\begin{align}
\mathscr{T}_{r,v}^{i}(p;F,A,L): & =\int_{(\partial A+p)_{\oplus r}}\ointop_{r\mathbb{S}^1(q)}F(p-u)\mathbf{1}_{B^{i}(p)}(u)\cdot\mathrm{d}uL\left(\mathrm{d}q\right);\label{auxfunctionalvorticity}\\
\mathscr{T}_{r,d}^{i}(p;F,A,L): & =\int_{(\partial A+p)_{\oplus r}}\ointop_{r\mathbb{S}^1(q)}F(p-\cdot)\mathbf{1}_{B^{i}(p)}\cdot n\mathrm{d}s\mathrm{d}L\left(\mathrm{d}q\right),\label{auxfunctionaldiv}
\end{align}
with $B^{1}=A$, $B^{2}=A^{*}$. Hence, we proceed to study these functionals in more detail. To improve the presentation of the following results, we introduce
a simplified version of Assumption \ref{assumptionambitset1}:

\begin{assumption}\label{assumption2}Let $A\subseteq\mathbb{R}^{2}$
	be a Jordan domain with Lipschitz-regular boundary satisfying that 
	\[
	\int_{\mathbb{S}^{1}}\left\lbrace  \sum_{q:(q,u)\in N(A)}(\delta_{A}(q,u)\wedge1)+\sum_{q:(q,u)\in N(A^{*})}(\delta_{A^{*}}(q,u)\wedge1)\right\rbrace\mathcal{H}^{1}(\mathrm{d}u)<\infty,
	\]
	and $F:\mathbb{R}^{2}\rightarrow\mathbb{R}^{2}$ a measurable function that is continuous on either $A\cap(\partial A)_{\oplus r_{0}}$ or
	$A^{*}\cap(\partial A)_{\oplus r_{0}}$ for some $r_{0}>0$.
	
\end{assumption}
As for the case of $\mathscr{C}_{r}(p;X)$ and $\mathscr{D}_{r}(p;X)$, we will only proof our statements for $\mathscr{T}_{r,d}^{i}(p;F,A,L)$.
\subsubsection*{Negligibility of $\partial\mathscr{C}_{r}(p;X)$ and $\partial\mathscr{D}_{r}(p;X)$
	in the case when $\left.F\right|_{-\partial\mathcal{R}}=0$ and $L$
	arbitrary}

\begin{theorem}Let $F:\mathbb{R}^{2}\rightarrow\mathbb{R}^{2}$ be
	such that $F(-\cdot)$ and $A\subseteq\mathbb{R}^{2}$ are as in Assumption
	\ref{assumption2}. Then for $r>0$ small, the functionals $\mathscr{T}_{r,v}^{i}(\cdot;F,A,L)$
	and $\mathscr{T}_{r,d}^{i}(\cdot;F,A,L)$ introduced in (\ref{auxfunctionalvorticity})
	and (\ref{auxfunctionaldiv}) are well defined. Furthermore, if $F(-\cdot)$
	is continuously differentiable on $B^{i}\cap(\partial A)_{\oplus r_{0}}$
	and $\left.F\right|_{-\partial A}=0$, then for any $p\in\mathbb{R}^{2}$,
	it holds that for $i=1,2$, as $r\downarrow0$,
	\begin{equation}
	r^{-2}\mathscr{T}_{r,v}^{i}(p;F,C,L)\overset{\mathbb{P}}{\rightarrow}0;\,\,\,\,r^{-2}\mathscr{T}_{r,d}^{i}(p;F,C,L)\overset{\mathbb{P}}{\rightarrow}0.\label{convboundaryFnull}
	\end{equation}
\end{theorem}

\begin{proof}We first note that by Lemma \ref{keylemmaasymptotics} and its subsequent remark in Appendix B, 
	\[ c_{1}=\sup_{r\leq r_{0,}q\in(\partial A)_{\oplus r}}\left|r^{-2}\ointop_{r\mathbb{S}^1(q)}F(-\cdot)\mathbf{1}_{B^{i}}\cdot n\mathrm{d}x\right|<\infty.\] 
	If $c_{1}=0$, then the result is trivial, so assume
	that $c_{1}>0$. An application of Lemma 2.1.5 in \cite{Rosinski07}
	gives us that for $i=1,2$
	\[ 
	\int_{(\partial A)_{\oplus r}}\Phi_{L}^{0}\left(r^{-2}\ointop_{r\mathbb{S}^1(q)}F(-\cdot)\mathbf{1}_{B^{i}}\cdot n\mathrm{d}x\right)\mathrm{d}q  \leq2Leb((\partial A)_{\oplus r})\Phi_{L}^{0}\left(c_{1}\right)\rightarrow0, \]
	which by stationarity and Theorem \ref{DCTstcint} in Appendix A, give all the conclusions
	in this theorem.\end{proof}

\begin{corollary}\label{corFvanishboundary}Let Assumption \ref{assumptionambitset1}
	holds. If $\left.F\right|_{-\mathcal{R}}=0$, then for any $p\in\mathbb{R}^{2}$,
	when $r\downarrow0$, we have that
	\[
	r^{-2}\partial\mathscr{C}_{r}(p;X)\overset{\mathbb{P}}{\rightarrow}0;\,\,\,\,r^{-2}\partial\mathscr{D}_{r}(p;X)\overset{\mathbb{P}}{\rightarrow}0.
	\]
\end{corollary}

\subsubsection*{Asymptotics of $\mathscr{T}_{r,v}^{i}(p;F,C,L)$ and $\mathscr{T}_{r,d}^{i}(p;F,C,L)$
	in the case when $\left.F\right|_{-\partial\mathcal{R}}\protect\neq0$ }

Let us now concentrate on the situation when $\left.F\right|_{-\partial\mathcal{R}}\protect\neq0$. Recall
that $v_{\beta}=2\{2\int_{0}^{1}(1-s^{2})^{\beta/2}\mathrm{d}s\}^{1/\beta}$
and that $(x,y)^{\perp}=(-y,x)$.

\begin{theorem}\label{Tauasymptotics}Let $F:\mathbb{R}^{2}\rightarrow\mathbb{R}^{2}$
	be such that $F(-\cdot)$ and $A\subseteq\mathbb{R}^{2}$ are as in
	Assumption \ref{assumption2}. Take $\mathscr{T}_{r,v}^{i}(\cdot;F,A,L)$
	and $\mathscr{T}_{r,d}^{i}(\cdot;F,A,L)$ as in (\ref{auxfunctionalvorticity})
	and (\ref{auxfunctionaldiv}), respectively, with $L$ a homogeneous
	L\'evy basis with characteristic triplet $(\gamma,b,\nu)$. Assume that
	$\left.F\right|_{-\partial A}\neq0$. Then, 
	\begin{enumerate}
		\item If $b>0$, 
		\begin{align*}
		\frac{1}{v_{2}r^{1+1/2}}\mathscr{T}_{r,v}^{i}(p;F,A,L) & \overset{\mathcal{F}\text{-}fd}{\longrightarrow}\int_{\partial A(p)}F(p-q)\cdot u_{A(p)}(q)W_{\mathcal{H}^{1}}(\mathrm{d}q);\\
		\frac{1}{v_{2}r^{1+1/2}}\mathscr{T}_{r,d}^{i}(\cdot;F,A,L) & \overset{\mathcal{F}\text{-}fd}{\longrightarrow}\int_{\partial A(p)}F(p-q)\cdot u_{A(p)}^{\perp}(q)W_{\mathcal{H}^{1}}(\mathrm{d}q);
		\end{align*}
		where $u_{A(p)}$ as in Assumption \ref{assumption2}, and $W_{\mathcal{H}^{1}}$
		as in Theorem \ref{gaussianattractor}.
		\item If $b=0$ and $\int_{\mathbb{R}}(1\land\left|x\right|)\nu(\mathrm{d}x)<\infty$,
		then as $r\downarrow0$
		\begin{align*}
		\frac{1}{\pi r^{2}}\mathscr{T}_{r,v}^{i}(p;A,L) & \overset{\mathbb{P}}{\rightarrow}(-1)^{i}\gamma_{0}\int_{\partial A(p)}F(p-q)\cdot u_{A(p)}(q)\mathcal{H}^{1}(\mathrm{d}q);\\
		\frac{1}{\pi r^{2}}\mathscr{T}_{r,d}^{i}(p;A,L) & \overset{\mathbb{P}}{\rightarrow}(-1)^{i}\gamma_{0}\int_{\partial A(p)}F(p-q)\cdot u_{A(p)}^{\perp}(q)\mathcal{H}^{1}(\mathrm{d}q);
		\end{align*}
		where $\gamma_{0}=-\int_{\left|x\right|\leq1}x\nu(\mathrm{d}x)$.
		\item If $L$ satisfies the assumptions of Theorem \ref{stableattractor},
		for some $2>\beta\geq1$, then, as $r\downarrow0$,
		\begin{align*}
		\frac{1}{v_{\beta}r^{1+1/\beta}}\mathscr{T}_{r,v}^{i}(p;F,A,L) & \overset{\mathcal{F}\text{-}fd}{\longrightarrow}(-1)^{i}\int_{\partial A(p)}F(p-q)\cdot u_{A(p)}(q)M_{\mathcal{H}^{1}}^{\beta}(\mathrm{d}q);\\
		\frac{1}{v_{\beta}r^{1+1/\beta}}\mathscr{T}_{r,d}^{i}(p;F,A,L) & \overset{\mathcal{F}\text{-}fd}{\longrightarrow}(-1)^{i}\int_{\partial A(p)}F(p-q)\cdot u_{A(p)}^{\perp}(q)M_{\mathcal{H}^{1}}^{\beta}(\mathrm{d}q);
		\end{align*}
		where $M_{\mathcal{H}^{1}}^{\beta}$ as in Theorem \ref{stableattractor}.
	\end{enumerate}
\end{theorem}

\begin{proof}The proof is divided
	in two steps. First we show the convergence of the finite-dimensional
	distributions. After that, we conclude the proof by showing the stable
	convergence. 
	
	\textbf{Finite-dimensional convergence: }We start by observing that
	if $L$ has characteristic exponent $\psi$ then by Theorem 2 in \cite{IvanovJ17}
	as $r\downarrow0$
	\[
	r\psi(r^{-1/\beta}z)\rightarrow\psi_{\beta}(z)=\begin{cases}
	-\frac{1}{2}b^{2}z^{2} & \text{if }b>0\text{ and }\beta=2;\\
	i\gamma_{0}z & \text{if }b=0,\int_{\mathbb{R}}(1\land\left|x\right|)\nu(\mathrm{d}x)<\infty\text{ and }\beta=1;\\
	\psi_{K_{\pm},\beta,\hat{\gamma}}(z) & \text{under }3.\text{ and }1\leq\beta<2;
	\end{cases}
	\]
	where $\psi_{K_{\pm},\beta,\hat{\gamma}}$ denotes the characteristic
	exponent of a strictly $\beta$-stable distribution whose parameters
	$K_{\pm}$ and $\hat{\gamma}$ are defined as in Theorem \ref{stableattractor}.
	Therefore, by Lemma \ref{keylemmaasymptotics-1} in Appendix B
	\[
	\frac{1}{v_{\beta}r^{1+1/\beta}}\mathscr{T}_{r,d}^{i}(p;F,A,L)\overset{fd}{=}\frac{1}{v_{\beta}r^{1+1/\beta}}\mathscr{T}_{r,d}^{i}(p;F,A,L_{\beta})+o_{\mathbb{P}}(1),
	\]
	where $\overset{fd}{=}$ stand by equality of the finite-dimensional
	distributions, and $L_{\beta}$ is a homogeneous L\'evy basis with characteristic
	exponent $\psi_{\beta}.$ Hence, it is enough to show that the asymptotics
	in 1.-3. holds for $\mathscr{T}_{r,d}^{i}(\cdot;F,A,L_{\beta})$.
	For $i=1,2$ , $z\in\mathbb{R}$, $r\leq r_{0}^{i}$, put
	\[
	z_{r}^{i}(q):=(v_{\beta}r)^{-1}z\ointop_{r\mathbb{S}^1(q)}F(-\cdot)\mathbf{1}_{B^{i}}\cdot n\mathrm{d}x\mathbf{1}_{(\partial A)_{\oplus r}}(q).
	\]
	By relation (\ref{cumulantstochint}) in Appendix A, and the strict stability of $\psi_{\beta}$,
	we have that
	\[
	\mathcal{C}\left[z\ddagger\frac{1}{v_{\beta}r^{1+1/\beta}}\mathscr{T}_{r,d}^{i}(p;F,A,L_{\beta})\right]=\frac{1}{r}\left\{ \int_{A_{\oplus r}\backslash A}\psi_{\beta}(z_{r}^{i}(q))\mathrm{d}q+\int_{A_{\oplus r}^{*}\backslash A^{*}}\psi_{\beta}(z_{r}^{i}(q))\mathrm{d}q\right\} .
	\]
	For the notation involved below, we refer to the reader to Appendix B. We want to apply Theorems \ref{lemmadiffeo} and \ref{lemmadiffeo-1} to the previous relation. From Assumption \ref{assumption2},
	we have that $\mathcal{H}^1(\partial^{++}A\cap\partial^{++}A^{*}\backslash\partial A)=0$.
	Consequently, in this case Theorems \ref{lemmadiffeo} and \ref{lemmadiffeo-1} read as
	\begin{align}
	\frac{1}{r}\int_{A_{\oplus r}\backslash A}\psi_{\beta}(z_{r}^{i}(q))\mathrm{d}q & =\int_{\partial A}\int_{0}^{1}\psi_{\beta}\left[z_{r}^{i}(q+rsu_{A}(q))\right]\mathbf{1}_{\{\delta_A(q,u)>rs\}}\mathrm{d}s\mathcal{H}^{1}(\mathrm{d}q)\label{eq:charcttripl}\\
	+&\frac{w_{1}}{r}\int_{N(A)}\left[\int_{0}^{r}s\psi_{\beta}\left[z_{r}^{i}(q+su_{A}(q))\right]\mathbf{1}_{\{\delta_A(q,u)>s\}}\mathrm{d}s\right]\mu_{0}(A;\mathrm{d}(q,u)).\nonumber 
	\end{align}
	Lemma \ref{keylemmaasymptotics} in Appendix B, guarantees that for almost all $q\in\partial A$, as $r\downarrow0$
	\[
	z_{r}^{i}(q+rsu_{A}(q))\rightarrow v_{\beta}^{-1}z(-1)^{i}2\sqrt{1-s^{2}}F(q)\cdot u_{A}(q).
	\]
	and
	\[
	\int_{0}^{1}\int_{\partial A}\left\Vert \psi_{\beta}(v_{2}^{-2}z_{r,v}^{i}(s,q))\mathbf{1}_{\{\delta_A(q,u(q))>rs\}}\right\Vert \mathcal{H}^{1}(\mathrm{d}q)\mathrm{d}s\leq K2\mathcal{H}^{1}(\partial A)<\infty,
	\]
	for some constant $K>0$. Hence, the Dominated Convergence Theorem
	can be applied to get that as $r\downarrow0$
	\[
	\int_{\partial A}\int_{0}^{1}\psi_{\beta}\left[z_{r}^{i}(q+rsu_{A}(q))\right]\mathbf{1}_{\{\delta_A(q,u)>rs\}}\mathcal{H}^{1}(\mathrm{d}q)\mathrm{d}s\rightarrow\frac{1}{2}\int_{\partial A}\psi_{\beta}[(-1)^{i}F(q)\cdot u_{A}(q)]\mathcal{H}^{1}(\mathrm{d}q).
	\]
	We claim that the last integral in (\ref{eq:charcttripl}) vanishes
	when $r\downarrow0$. Indeed, as before we can choose $K>0$ such
	that
	\begin{align*}
	\left|\frac{1}{r}\int_{N(A)}\int_{0}^{r}s\psi_{\beta}\left[z_{r}^{i}(q+su_{A}(q))\right]\mathbf{1}_{\{\delta_A(q,u)>s\}}\mathrm{d}s\mu_{0}(A;\mathrm{d}(q,u))\right|&\leq \frac{K}{r}\int_{N(A)}\int_{0}^{r}s\mathbf{1}_{\{\delta_A(q,u)>s\}}\\
	\times&\mathrm{d}s\left|\mu_{0}\right|(A;\mathrm{d}(q,u)).
	\end{align*}
	On the other hand, for any $r\leq1$
	\[
	\frac{1}{r}\int_{0}^{r}s\mathbf{1}_{\{\delta_A(q,u)>s\}}\mathrm{d}s\leq [1\wedge\delta_A(q,u)]^2+1\wedge\delta_A(q,u).
	\]
	Our claim then follows by the integrability condition in Assumption \ref{assumption2}, Theorem \ref{lemmadiffeo} in the appendix,
	and the Dominated Convergence Theorem. In a similar way it is possible
	to verify that as $r\downarrow0$
	\[
	\frac{1}{r}\int_{A_{\oplus r}^{*}\backslash A^{*}}\psi_{\beta}(z_{r}^{i}(q))\mathrm{d}q\rightarrow\frac{1}{2}\int_{\partial A}\psi_{\beta}[(-1)^{i}F(q)\cdot u_{A}(q)]\mathcal{H}^{1}(\mathrm{d}q).
	\]
	All above give us the pointwise convergence in 1.-3. The finite-dimensional
	convergence can be shown using similar arguments as above as well
	as an application of the Cram\'er\textendash Wold methodology and the Inclusion-Exclusion
	principle.
	
	\textbf{Stable convergence: }To avoid extra notation we only show
	that the stable convergence holds when $b>0$. For the other cases a similar
	argument can be applied.
	
	Let $A$ be a bounded Borel set. Since for $i=1,2$, $\mathscr{T}_{r,d}^{i}(\cdot;F,A,L)$
	is $\mathcal{F}_{L}$-measurable and thanks to Theorem 3.2 in \cite{HauslerLuschgy15},
	it is sufficient for the $\mathcal{F}$-stable convergence of the
	finite-dimensional distributions of $\{\frac{1}{v_{2}r^{1+1/2}}\mathscr{T}_{r,d}^{i}(p;F,C,L)\}_{p\in\mathbb{R}^{2}}$
	to show that for any $p_{1},\ldots,p_{m}$, $	(\{\frac{1}{v_{\beta}r^{1+1/}}\mathscr{T}_{r,d}^{i}(p_{j};F,A,L)\}_{j=1}^{m},L(A))$ converges in distribution towards
	\[
	\left(\{\int_{\partial A+p_{j}}F(p_{j}-q)\cdot u_{A(p_{j})}(q)W_{\mathcal{H}^{1}}(\mathrm{d}q)\}_{j=1}^{m},L(A)\right).
	\]
	For $z=(z_{1},\ldots,z_{m+1})\in\mathbb{R}^{m+1}$, let $\mathcal{C}\left[z\ddagger\{\frac{1}{v_{2}r^{1+1/2}}\mathscr{T}_{r,d}^{i}(p_{j};F,A,L)\}_{j=1}^{m},L(A)\right]$
	be the characteristic exponent of the random vector $(\{\frac{1}{v_{2}r^{1+1/2}}\mathscr{T}_{r,d}^{i}(p_{j};F,A,L)\}_{j=1}^{m},L(A))$.
	For the former, we are going to show that for $z\neq0$, it converges as $r\downarrow0$ to
	\begin{equation}
	\mathcal{C}\left(z\ddagger\{\int_{\partial A+p_{j}}F(p_{j}-q)\cdot u_{A(p_{j})}(q)W_{\mathcal{H}^{1}}(\mathrm{d}q)\}_{j=1}^{m}\right)
	+\mathcal{C}(z_{m+1}\ddagger L(A)).\label{stableconvergenceGaussian}
	\end{equation}
	Let $\mathcal{A}_{m,r}=\cup_{i=1}^{m}(\partial A+p_{j})_{\oplus r}$.
	If $A\cap(\partial A+p_{j})=\emptyset$, for any $i=1,\ldots,m$,
	then for $r$ small enough $\mathcal{A}_{m,r}\cap A=\emptyset$ and
	(\ref{stableconvergenceGaussian}) follows by independently scattered
	property of $L$. Suppose then that $\mathcal{A}_{m,r}\cap A\neq\emptyset.$
	Then, almost surely 
	\begin{align*}
	L(A) & =L(A\cap\mathcal{A}_{m,r})+L(A\backslash\mathcal{A}_{m,r}).
	\end{align*}
	Once again, in view that $L$ is independently scattered, we get that \linebreak $\mathcal{C}\left[z\ddagger\{\frac{1}{v_{2}r^{1+1/2}}\mathscr{T}_{r,d}^{i}(p_{j};F,A,L)\}_{j=1}^{m},L(A)\right]$ equals
	\begin{align*}
	\log\mathbb{E}\left\{ i\left[\sum_{j=1}^{m}z_{i}\frac{1}{v_{2}r^{1+1/2}}\mathscr{T}_{r,d}^{i}(p_{j};F,C,L)+z_{m+1}L(A\cap\mathcal{A}_{m,r})\right]\right\} 
	+\mathcal{C}(z_{m+1}\ddagger L(A\backslash\mathcal{A}_{m,r})).
	\end{align*}
	Since $L(A\cap\mathcal{A}_{m,r})\overset{\mathbb{P}}{\rightarrow}0$
	and $L(A\backslash\mathcal{A}_{m,r})\overset{\mathbb{P}}{\rightarrow}L(A)$,
	equation (\ref{stableconvergenceGaussian}) follows by the previous
	relation and Slutsky's Theorem.\end{proof}

\subsection{Proof of Theorems \ref{gaussianattractor}-\ref{stableattractor} }

In this part we present the proof of Theorems \ref{gaussianattractor}-\ref{stableattractor}
by combining the results obtained in the previous subsections. 

\begin{proof} Firstly observe that in general, if $L$ has characteristic triplet $(\gamma,b,\nu)$, then from (\ref{DandCnodrift}), Lemma \ref{lemmadecomposition} and the L\'evy-Ito\^o decomposition, we get that a.s, 
	\begin{equation}
	\mathscr{D}_{r}(p;X)=\partial\mathscr{D}_{r}(p;\tilde{X})+\mathring{\mathscr{D}}_{r}(p;\tilde{X}),\label{decompositionproofmainthm}
	\end{equation}	
	where 
	\[
	\tilde{X}(p)=\int_{\mathcal{R}+p}F(p-q)\tilde{L}(\mathrm{d}q),
	\]
	and $\tilde{L}$ as in Theorem \ref{divergenceandvorticitytheorem}.	
	Assume now that $\left.F\right|_{-\partial\mathcal{R}}=0$. In this case, by Corollary \ref{corFvanishboundary},
	we have that
	\begin{equation}
	\mathscr{D}_{r}(p;X)=\mathring{\mathscr{D}}_{r}(p;\tilde{X})+o_{\mathbb{P}}(r^2)\label{classicalrate}.
	\end{equation}
	The convergence in Theorem \ref{divergenceandvorticitytheorem} follows immediately from this and Proposition \ref{asymptoticsint}. Suppose now that $\left.F\right|_{-\partial\mathcal{R}}\neq0$. If $b=0$ and $\int_{\mathbb{R}}(1\land\left|x\right|)\nu(\mathrm{d}x)<\infty$, by equation (\ref{decompositionTauD}) and Theorem \ref{Tauasymptotics}, we still get that relation (\ref{classicalrate}) holds in this case and subsequently to the conclusion of Theorem \ref{divergenceandvorticitytheorem}.
	Now suppose that $b>0$. Similar arguments used in the proof of Theorem \ref{Tauasymptotics}
	\[
	\mathscr{D}_{r}(p;X)\overset{fd}{=}\partial\mathscr{D}_{r}(p;\tilde{X}^W)+o_{\mathbb{P}}(r^{1+1/2}),
	\]
	where $\tilde{X}^W$ is defined as $\tilde{X}$ but we replace $\tilde{L}$ by a Gaussian L\'evy basis with variance $b^2$. Another application of (\ref{decompositionTauD}) and Theorem \ref{Tauasymptotics} conclude the proof of Theorem \ref{gaussianattractor}.
	Finally, let the assumptions of Theorem \ref{stableattractor} hold. Analogously as the preeceding argument, we deduce that if $1<\beta<2$
	\[
	\mathscr{D}_{r}(p;X)\overset{fd}{=}\partial\mathscr{D}_{r}(p;\tilde{X}^\beta)+o_{\mathbb{P}}(r^{1+1/\beta}),
	\]
	while for $1=\beta$
	\[
	\mathscr{D}_{r}(p;X)\overset{fd}{=}\partial\mathscr{D}_{r}(p;\tilde{X}^\beta)+\mathring{\mathscr{D}}_{r}(p;\tilde{X})+o_{\mathbb{P}}(r^2),
	\]
	where $\tilde{X}^W$ is defined as $\tilde{X}$ but  $\tilde{L}$ is substituted by a strictly $\beta$-stable L\'evy basis with the parameters given in Theorem \ref{stableattractor}. The conclusions of Theorem \ref{stableattractor} then follows from this approximation and Theorem \ref{Tauasymptotics}. \end{proof}

\subsection{Proof of Theorem \ref{thmdivvortambit}}

As an application of Theorem \ref{divergenceandvorticitytheorem}
and Proposition \ref{convstchintambittype}  in
Appendix A, we present a proof for
Theorem \ref{thmdivvortambit}.

\begin{proof}
	Note first that from Proposition \ref{propisotropyambit} and Remark
	\ref{remarkgenambitfieldsproperties}, we can find a measurable modification of $Y$, which will be also
	denoted by $Y$, satisfying that 
	\[
	\mathbb{E}\left[\int_{0}^{2\pi}\left| Y(p+ru(\theta))\cdot u(\theta) \right|\mathrm{d}\theta\right]\leq2\pi\mathbb{E}[\left\Vert Y(0)\right\Vert ]<\infty,
	\]
	meaning that the field $\mathscr{D}_{r}(p;Y)$ is well defined for
	any $p\in\mathbb{R}^{2}$ and $r>0$. The key step for the rest of
	the proof consists in showing that 
	\begin{equation}
	\left\{ \mathscr{D}_{r}(p;Y)\right\} _{p\in\mathbb{R}^{2}}\overset{d}{=}\{\mathring{\mathscr{D}}_{r}(p;Y)+\partial\mathscr{D}_{r}(p;Y)\}_{p\in\mathbb{R}^{2}},\label{fubiniversion}
	\end{equation}
	where $\mathring{\mathscr{D}}_{r}(p;Y)$ and $\partial\mathscr{D}_{r}(p;Y)$
	are defined as in (\ref{lemmadecomposition}) when we replace $L(\mathrm{d}q)$
	by $V(q)L(\mathrm{d}q)$. Indeed, if this were true, under our assumptions, Proposition \ref{convstchintambittype} in
 Appendix A
	can be applied to deduce that the limit in probability of $(\pi r^{2})^{-1}\mathscr{D}_{r}(p;Y)$
	exists and it is the same as the one of $(\pi r^{2})^{-1}\mathring{\mathscr{D}}_{r}(p;X)$ when
	we replace $L(\mathrm{d}q)$ by $V(q)L(\mathrm{d}q)$,
	which would complete the proof. 
	
	To show (\ref{fubiniversion}), first observe that due to the stationarity
	of $V$ (as well as its square integrability), the conditions of Lemma
	\ref{Fubinilemma} in
	Appendix B are satisfied almost surely for $K=[0,2\pi]$,
	$\varphi(\theta)=p+ru(\theta)$, $f_{\mathscr{D}}(\varphi(\theta),q):=\mathbf{1}_{\mathcal{R}+p} F(\varphi(\theta)-q)\cdot u(\theta) V(q)$
	and $\left|D\varphi(\theta)\right|=r>0$. Consequently, conditioned
	to $V$, for any $p\in\mathbb{R}^{2}$ a.s. 
	\[
	\mathscr{D}_{r}(p;Y)\left|V\right.=\mathring{\mathscr{D}}_{r}(p;Y)\left|V\right.+\partial\mathscr{D}_{r}(p;Y)\left|V\right..
	\]
	Relation (\ref{fubiniversion}) then follows immediately from this.\end{proof}

\begin{remark}Unlike in the case when $V$ is identically constant, $\mathscr{D}_{r}(p;Y)$ does not vanishes when $L$ is deterministic. This is the reason why the limiting fields in Theorem \ref{divergenceandvorticitytheorem} include the whole $L$ and not $\tilde{L}$ as in Theorem \ref{divergenceandvorticitytheorem}. Finally we note that Theorem \ref{thmdivvortambit} remains valid in the case when (\ref{bigjumpsVL}) is replaced by $\int_{0}^{2\pi}\left| Y(p+ru(\theta))\cdot u(\theta)\right|\mathrm{d}\theta<\infty$,
	almost surely.
	
\end{remark}

\subsection{Final remarks and generalizations\label{subsec:Final-remarks-and}}

To conclude this section we further discuss other possible asymptotic
rates for $\mathscr{D}_{r}(p;X)$. We
also briefly discuss some generalizations on the ambit set. First
observe that the only possible limit fields for $\mathscr{D}_{r}(p;X)$
are the one appearing in Theorems \ref{gaussianattractor}-\ref{stableattractor}.
Indeed, we have established in Proposition \ref{asymptoticsint},
that in general
\[ 
(\pi r^{2})^{-1}\mathscr{\mathring{D}}_{r}(p;X)\overset{\mathbb{P}}{\rightarrow} \int_{\mathcal{R}+p}\nabla\cdot F(p-q)L(dq).
\]
Thus, in the framework of Assumption \ref{assumptionambitset1} ($n=1$)
let $a_{r}\rightarrow0$ as $r\downarrow0$. Then 
\[
\mathcal{C}\left[z\ddagger\frac{\partial\mathscr{D}_{r}(p;X)}{ra(r)}\right]=\frac{1}{r}\int_{(\partial A)_{\oplus r}}r\psi(a_{r}^{-1}z_{r}^{i}(q))\mathrm{d}s\mathrm{d}q,
\]
where $z_{r}^{i}(q)$ is as in the proof of Theorem \ref{Tauasymptotics}.
We have shown that $r^{-1}Leb((\partial A)_{\oplus r})$
and $z_{r}^{i}(q)$ are uniformly bounded. Thus, we deduce that the
sequence $\frac{\partial\mathscr{D}_{r}(p;X)}{ra(r)}$ is bounded
in probability whenever $r\psi(a_{r}^{-1}\cdot)$ is. This is achieved
for instance when $r\psi(a_{r}^{-1}\cdot)$ is convergent, which according
to \cite{IvanovJ17}, occurs if and only if $a_{r}=r^{1/\beta}l_{r}$,
for some $0<\beta\leq2$ and a slowly varying function $l_{r}$ at
$0$. In this framework, Lemma \ref{keylemmaasymptotics-1} in
Appendix B remains
valid, leading this to the conclusion that Theorems \ref{gaussianattractor}-\ref{stableattractor}
still hold when we replace $r^{1/\beta}$ by $a_{r}$.

Let us now discuss some feasible generalizations that can be considered
for further research. Our proofs are based on two key results, namely Theorem \ref{lemmadiffeo} and
Lemma \ref{keylemmaasymptotics},  in Appendix A and B, respectively. The
former deals with the asymptotic behaviour of the mapping $q\mapsto\ointop_{r\mathbb{S}^1(q)}F\mathbf{1}_{B^{i}}\cdot n\mathrm{d}x$
on $(\partial A)_{\oplus r}$ when $r\downarrow0$ which is in general extremely
dependent on the geometry/smoothness of $\partial A$. Thus, a natural
generalization of our framework is to consider non-smooth
curves, e.g. fractals. Note that assumption \ref{assumptionambitset1} implies in particular that
the limit $\lim_{r\downarrow0}r^{-1}Leb((\partial A)_{\oplus r})$
exists and is finite. The later property is called 1-dimensional {\it  Minskowki measurability}.
More generally, a set $A\subseteq\mathbb{R}^{d}$ is said to be $s$-dimensional
Minkowski measurable if there is $s\leq d$ such that the limit $\lim_{r\downarrow0}r^{-(d-s)}Leb(A{}_{\oplus r})$
exists, is finite and different from zero. Therefore, if $\partial A$ is $s$-dimensional Minkowski measurable, then
we can deduce 
\[
\partial\mathscr{D}_{r}(p;X)=O_{\mathbb{P}}(r^{1+\frac{d-s}{\beta}}).
\]
Minkowki measurability holds for a big
class of self similar curves (see \cite{Gatz00}). The main challenge
in this framework relies on the identification of the limit (in case
it exists) of $\ointop_{r\mathbb{S}^1(q)}F(-\cdot)\mathbf{1}_{B^{i}}\cdot n\mathrm{d}x\mathbf{1}_{(\partial A)_{\oplus r}}$.

%%%%%%%%%%%%%%%%%%%%%%%%%%%%%%%%%%%%%%%%%%%%%%%%%%%%%%%%%%%%%%%%%%%%%%%%%%%%%%%%%%%%%%%%%%%%%%%%
%%%%%%%%%%%%%%%%%%%%%%%%%%%%%%%%%%%%%%%%%%%%%%%%%%%%%%%%%%%%%%%%%%%%%%%%%%%%%%%%%%%%%%%%%%%%%%%%
%%%%%%%%%%%%%%%%%%%%%%%%%%%%%%%%%%%%%%%%%%%%%%%%%%%%%%%%%%%%%%%%%%%%%%%%%%%%%%%%%%%%%%%%%%%%%%%%
%%%%%%%%%%%%%%%%%%%%%%%%%%%%%%%%%%%%%%%%%%%%%%%%%%%%%%%%%%%%%%%%%%%%%%%%%%%%%%%%%%%%%%%%%%%%%%%%

%%%%%%%%%%%%%%%%%%%%%%%%%%%%%%%%%%%%%%%%%%%%%%%%%%%%%%%%%%%%%%%%%%%%%%%%%%%%%%%%%%%%%%%%%%%%%%%%
%%%%%%%%%%%%%%%%%%%%%%%%%%%%%%%%%%%%%%%%%%%%%%%%%%%%%%%%%%%%%%%%%%%%%%%%%%%%%%%%%%%%%%%%%%%%%%%%
%%%%%%%%%%%%%%%%%%%%%%%%%%%%%%%%%%%%%%%%%%%%%%%%%%%%%%%%%%%%%%%%%%%%%%%%%%%%%%%%%%%%%%%%%%%%%%%%
%%%%%%%%%%%%%%%%%%%%%%%%%%%%%%%%%%%%%%%%%%%%%%%%%%%%%%%%%%%%%%%%%%%%%%%%%%%%%%%%%%%%%%%%%%%%%%%%
%%%%%%%%%%%%%%%%%%%%%%%%%%%%%%%%%%%%%%%%%%%%%%%%%%%%%%%%%%%%%%%%%%%%%%%%%%%%%%%%%%%%%%%%%%%%%%%%
%%%%%%%%%%%%%%%%%%%%%%%%%%%%%%%%%%%%%%%%%%%%%%%%%%%%%%%%%%%%%%%%%%%%%%%%%%%%%%%%%%%%%%%%%%%%%%%%
%%%%%%%%%%%%%%%%%%%%%%%%%%%%%%%%%%%%%%%%%%%%%%%%%%%%%%%%%%%%%%%%%%%%%%%%%%%%%%%%%%%%%%%%%%%%%%%%
%%%%%%%%%%%%%%%%%%%%%%%%%%%%%%%%%%%%%%%%%%%%%%%%%%%%%%%%%%%%%%%%%%%%%%%%%%%%%%%%%%%%%%%%%%%%%%%%

\section*{Acknowledgements}
The author wishes to thank Ole E. Barndorff-Nielsen and J\"urgen Schmiegel
for fruitful ideas that inspired this paper. The class of models introduced
in Section 4 came out from private discussions with them. The author
is also grateful to Gerardo Gonz\'alez and Markus Kiderlen for helpful
comments on a previous draft of this paper.

\section*{Appendix A}

For a self-contained presentation, we present in this appendix some
results related to the stochastic integration with respect to L\'evy
bases and formulas for the parallel sets of a closed set known as
Steiner formula.

\subsection*{A Steiner-type formula for closed setss}
This appendix gives a Steiner formula for closed sets in terms of the so-called
{\it reach measures} of $A$. We refer to \cite{HugGuntWel04} for more details. Such Steiner formula reads as follow:
\begin{theorem}\label{lemmadiffeo}For any non-empty closed set $A\subseteq\mathbb{R}^{d},$
	there exist uniquely determined reach measures $\mu_{0}(A;\cdot),\ldots,\mu_{0}(A;\cdot)$
	defined on $N(A)$ satisfying that for all $j=0,\ldots,d-1$, $r>0$
	and any compact set $B\subseteq\mathbb{R}^{d}$ 
	\[
	\int_{N(A)}\mathbf{1}_{B}(x)(r\wedge\delta_A(q,u))^{d-j}\left|\mu_{j}\right|(A;\mathrm{d}(q,u))<\infty.
	\]
	Moreover, if $w_{j}$ denotes the surface area of $\in\mathbb{S}^{j}$ and  $f:\mathbb{R}^{d}\rightarrow\mathbb{R}$
	is measurable with compact support, it holds
	\[ 
	\int_{A_{\oplus r}\backslash A}f(x)\mathrm{d}x  =\sum_{j=0}^{d-1}w_{d-i}\int_{N(A)}\int_{0}^{r}s^{d-1-i}\mathbf{1}_{\{\delta_A(q,u)>t\}} \times f(q+su)\mathrm{d}s\mu_{j}(A;\mathrm{d}(q,u)). \]
\end{theorem}

The reach measures $\mu_{d-1}(A;)$ and $\mu_{0}(A;)$ can be written
in an explicit way. To do this, we introduce some extra notation.
The {\it positive boundary} of $A$ is defined as
\[
\partial^{+}A:=\{q\in\partial A:(q,u)\in N(A)\},
\]
and setting $N(A,q):=\{u\in\mathbb{S}^{d-1}:(q,u)\in N(A)\}$, we write 
\[
\partial^{++}A:=\{q\in\partial^{+}A:N(A,q)=\{u(q)\} \text{ or } N(A,q)=\{u(q),-u(q)\};u(q)\in\mathbb{S}^{d-1}\},
\]

\begin{theorem}\label{lemmadiffeo-1}For any non-empty closed set
	$A\subseteq\mathbb{R}^{d},$ it holds that for any measurable and
	bounded function $g:N(A)\rightarrow\mathbb{R}$ with compact support,
	\begin{align*}
	\int_{N(A)}g(q,u)\mu_{d-1}(A;\mathrm{d}(q,u)) & =\frac{1}{2}\int_{\partial^{++}A}\sum_{u: (u,q)\in N(A)}g(q,u)\mathcal{H}^{d-1}(\mathrm{d}q);\\
	w_d\int_{N(A)}g(q,u)\mu_{0}(A;\mathrm{d}(q,u)) & =\int_{\mathbb{S}^{d-1}}\sum_{x:(u,q)\in N(A)}g(q,u)(-1)^{j_A(q,u)}\mathcal{H}^{d-1}(\mathrm{d}u),
	\end{align*}
	for a measurable function taking values in $\{0,1,\ldots d-1\}$.
\end{theorem}

\subsection*{Limits for sequences of stochastic integrals w.r.t. L\'evy bases.}
Below we will present some results concerning the existence and the
convergence of sequences of stochastic integrals with respect to L\'evy
bases. We refer the reader to \cite{RajputRosinski89}, \cite{BasseGravPed13},
and \cite{ChongKlupp15}. Fix a filtered probability space $(\Omega,\mathcal{F},(\mathcal{F}_{t})_{t\in\mathbb{R}},\mathbb{P})$
satisfying the usual conditions.
Recall that a function $\tau:\mathbb{R}^{m}\rightarrow\mathbb{R}^{m}$
is said to be a truncation function if it is bounded and $\tau(x)=x$
in a neighborhood of $0$. Denote by $\mathcal{B}_{b}\left(\mathbb{R\times R}^{d}\right)$
the bounded Borel sets on $\mathbb{R\times R}^{d}$. Let $\left(L\left(A\right):A\in\mathcal{B}_{b}\left(\mathbb{R\times R}^{d}\right)\right)$
be a real-valued homogeneous L\'evy basis with characteristic triplet
$\left(\gamma_{\tau},b,\nu\right)$ relative to a continuous truncation
function $\tau$, that is, the L\'evy seed of $L$ satisfies that 
\[
\mathcal{C}(z\ddagger L')=i\gamma_{\text{\ensuremath{\tau}}}z-\frac{1}{2}b^{2}z^{2}+\int_{\mathbb{R}\backslash\{0\}}[e^{izx}-1-iz\tau(x)]\nu(\mathrm{d}x).
\]
A real-valued random field of the form
\begin{equation}
\xi(\omega,s,q)=\sum\limits _{i=1}^{n}a_{i}\mathbf{1}_{F_{i}}(\omega)\mathbf{1}_{(u_{i},t_{i}]}(s)\mathbf{1}_{A_{i}}(q),\label{eq:simplerandomfield}
\end{equation}
where $A_{i}\in\mathcal{B}_{b}\left(\mathbb{R}^{d}\right)$, $u_{i}<t_{i}$,
$F_{i}\in\mathcal{F}_{u_{i}}$, and $a_{i}\in\mathbb{R},$ is called
a simple predictable random field. More generally, if $\mathcal{P}$
denotes the predictable $\sigma$-algebra associated to $\left(\mathcal{F}_{t}\right)_{t\in\mathbb{R}}$
a random field is said to be predictable if it is $\mathcal{P\otimes B}\left(\mathbb{R}^{d}\right)$-measurable
\footnote{In the case when there is no temporal component, $\mathcal{P}$ is
	replaced by $\mathcal{F}$ and all the results presented in this appendix
	remain valid.}. When $\xi$ is a simple random field as in (\ref{eq:simplerandomfield}),
the stochastic integral of $\xi$ w.r.t. $L$ is defined as 
\begin{equation}
I_{L}\left(\xi\right):=\int_{\mathbb{R}}\int_{\mathbb{R}^{d}}\xi\left(s,q\right)L\left(\mathrm{d}q\mathrm{d}s\right):=\sum\limits _{i=1}^{n}a_{i}\mathbf{1}_{F_{i}}L\left((u_{i},t_{i}]\times A_{i}\right).\label{eqn1.1}
\end{equation}
In stochastic integration theory one is usually looking for a linear
extension of $I_{L}$ into a rich enough linear subspace of predictable
random fields, let's say $\mathcal{D}\left(I_{L}\right)$, such that
$I_{L}\left(\xi\right)$ can be approximated by simple stochastic
integrals of the form (\ref{eqn1.1}) as well as satisfying a Dominated
Convergence Theorem, that is, if $\left(\xi_{n}\right)_{n\in\mathbb{N}}$
is a sequence of simple functions such that $\xi_{n}\rightarrow\xi$
point-wise, then $I_{L}\left(\xi_{n}\right)\overset{\mathbb{P}}{\rightarrow}I_{L}\left(\xi\right).$
Moreover, if $\left(\xi_{n}\right)_{n\in\mathbb{N}}\subset\mathcal{D}\left(I_{L}\right)$
such that $\xi_{n}\rightarrow\xi$ point-wise and $\left\vert \xi_{n}\right\vert \leq\xi^{\ast}$
for some $\xi^{\ast}\in\mathcal{D}\left(I_{L}\right)$, then $I_{L}\left(\xi_{n}\right)\overset{\mathbb{P}}{\rightarrow}I_{L}\left(\xi\right)$.
When $c\left(\left\{ s\right\} \times A\right)=0$ for every $t\in\mathbb{R}$,
we will choose $\mathcal{D}\left(I_{L}\right)$ to be the space of
predictable random fields such that almost surely
\begin{eqnarray}
1.\text{ } &  & \int_{\mathbb{R}}\int_{\mathbb{R}^{d}}\left\vert \gamma\xi\left(s,q\right)+\int_{\mathbb{R}}\left[\tau\left(\xi\left(s,q\right)x\right)-\xi\left(s,q\right)\tau\left(x\right)\right]\nu\left(\mathrm{d}x\right)\right\vert \mathrm{d}q\mathrm{d}s<\infty;\nonumber \\
2.\text{ } &  & \int_{\mathbb{R}}\int_{\mathbb{R}^{d}}\xi^{2}\left(s,q\right)b^{2}\mathrm{d}q\mathrm{d}s<\infty;\label{eqn1.5}\\
3.\text{ } &  & \int_{\mathbb{R}}\int_{\mathbb{R}^{d}}\int_{\mathbb{R}}\left(1\wedge\left\vert \xi\left(s,q\right)x\right\vert ^{2}\right)\nu\left(\mathrm{d}x\right)\mathrm{d}q\mathrm{d}s<\infty.\nonumber 
\end{eqnarray}
where $\tau$ is a continuous truncation function. \cite{ChongKlupp15}
have shown that $\mathcal{D}\left(I_{L}\right)$ is actually the biggest
closed linear subspace of predictable random fields in which $I_{L}$
can be defined in the previously explained sense. In the case when
$\xi$ is deterministic, we have that $I_{L}\left(\xi\right)$ is
ID, and
\begin{equation}
\mathcal{C}\left[z\ddagger\int_{\mathbb{R}}\int_{\mathbb{R}^{d}}\xi\left(s,q\right)L\right]\left(\mathrm{d}q\mathrm{d}s\right):=\int_{\mathbb{R}}\int_{\mathbb{R}^{d}}\mathcal{C}(z\xi\left(s,q\right)\ddagger L')\mathrm{d}q\mathrm{d}s,\,\,\,z\in\mathbb{R}.\label{cumulantstochint}
\end{equation}
Put 
\begin{align}
\Phi_{L}^{0}\left(y\right):= & U_{\tau}(y)+b^{2}y^{2}+\int_{\mathbb{R}}(1\wedge\left\vert yx\right\vert ^{2})\nu(\mathrm{d}x),\,\,\,y\geq0;\label{eq:deterministicmodular}
\end{align}
where
\begin{align}
U_{\tau}(y):= & \left\vert y\gamma_{\tau}+\int_{\mathbb{R}}\left[\tau\left(yx\right)-y\tau\left(x\right)\right]\nu(\mathrm{d}x)\right\vert ,\label{eq:modularU}
\end{align}
and for $y<0$ let $\Phi_{L}^{0}\left(y\right)=\Phi_{L}^{0}\left(-y\right)$.
$\mathcal{D}\left(I_{L}\right)$ consists of those predictable fields
that almost surely $\int_{\mathbb{R}}\int_{\mathbb{R}^{d}}\Phi_{L}^{0}(\xi\left(s,q\right))\mathrm{d}q\mathrm{d}s<\infty$. 
Now, we proceed to describe the topological structure of $\mathcal{D}\left(I_{L}\right)$.
Define
\begin{equation}
\Psi_{0}\left(\xi\right):=\mathbb{E}\left\{ \left[\int_{\mathbb{R}}\int_{\mathbb{R}^{d}}\Phi_{L}^{0}\left(\xi\left(s,q\right)\right)\mathrm{d}q\mathrm{d}s\right]\wedge1\right\} \text{, \ \ }\xi\in\mathcal{D}\left(I_{L}\right)\text{.}\label{eq:randommodular}
\end{equation}
The next result corresponds to Theorem 3.3 in \cite{BasseGravPed13}
accommodated to the multi-parameter case.
\begin{theorem}[\cite{BasseGravPed13}] \label{DCTstcint}Fix a continuous
	truncation function $\tau$ and let $L$ be a homogeneous L\'evy basis
	with characteristic triplet $\left(\gamma_{\tau},b,\nu\right)$. Suppose
	that $\left(\xi_{n}\right)_{n\in\mathbb{N}}\subset\mathcal{D}\left(I_{L}\right)$.
	Then as $n\rightarrow\infty$
	\begin{equation}
	\int_{\mathbb{R}}\int_{\mathbb{R}^{d}}\xi_{n}\left(s,q\right)L\left(\mathrm{d}q\mathrm{d}s\right)\overset{\mathbb{P}}{\rightarrow}0\text{ \ if and only if \ }\Psi_{0}\left(\xi_{n}\right)\rightarrow0.\label{eqn1.2}
	\end{equation}
	In particular, if the $\xi_{n}$'s are deterministic, this is equivalent
	to having that \linebreak $\int_{\mathbb{R}}\int_{\mathbb{R}^{d}}\Phi_{L}^{0}\left[\xi_{n}\left(s,q\right)\right]\mathrm{d}q\mathrm{d}s\rightarrow0.$
\end{theorem}
Based on the previous theorem, it is possible to find a sufficient
condition for the convergence of sequences of the ambit-type. To do
that the following lemma is crucial and it was originally stated in
Lemma 2.1.5 in \cite{Rosinski07}. 
\begin{lemma} \label{boundsrosinski} For $\tau(x)=\frac{x}{1\lor\left|x\right|}$
	let $\Phi_{L}^{0}$ be as in (\ref{eq:deterministicmodular}). Then
	$\Phi_{L}^{0}$ is continuous, even, and satisfies 
	\[
	\Phi_{L}^{0}\left(x+y\right)\leq3[\Phi_{L}^{0}\left(x\right)+\Phi_{L}^{0}\left(y\right)],
	\]
	and
	\[
	\Phi_{L}^{0}(Kx)\leq(K^{2}\lor2)\Phi_{L}^{0}(x),
	\]
	for any $x,y,K\in\mathbb{R}$.	
\end{lemma}
Recall that a random field $V$ is said to be bounded in $\mathcal{L}^{p}\left(\Omega,\mathcal{F},\mathbb{P}\right)$
for $p>0$, if $\sup_{s,q}\mathbb{E}\left(\left\vert V_{s}\left(q\right)\right\vert ^{p}\right)<\infty$.
Under this terminology we have the following result.
\begin{proposition} \label{convstchintambittype}Let $\tau(x)=\frac{x}{1\lor\left|x\right|}$
	and consider $L$ to be a homogeneous L\'evy basis with characteristic
	triplet $\left(\gamma_{\tau},b,\nu\right)$. Put 
	\[
	\xi_{n}\left(s,q\right):=f_{n}\left(s,q\right)V_{s}\left(q\right)\text{, \ \ }\left(s,q\right)\in\mathbb{R\times R}^{d}\text{,}
	\]
	where $\left(f_{n}\right)_{n\in\mathbb{N}}$ is a sequence of deterministic
	functions, and $V$ a predictable random field which is bounded in
	$\mathcal{L}^{2}\left(\Omega,\mathcal{F},\mathbb{P}\right)$. Then
	$\left(\xi_{n}\right)_{n\in\mathbb{N}}\subset\mathcal{D}\left(I_{L}\right)$
	provided that $\left(f_{n}\right)_{n\in\mathbb{N}}\subset\mathcal{D}\left(I_{L}\right)$.
	Moreover, as $n\rightarrow\infty$, $\int_{\mathbb{R}}\int_{\mathbb{R}^{d}}\xi_{n}\left(s,q\right)L\left(\mathrm{d}q\mathrm{d}s\right)\overset{\mathbb{P}}{\rightarrow}0$
	provided that $\int_{\mathbb{R}}\int_{\mathbb{R}^{d}}f_{n}\left(s,q\right)L\left(\mathrm{d}q\mathrm{d}s\right)\overset{\mathbb{P}}{\rightarrow}0$.	
\end{proposition}
\begin{proof}By Lemma \ref{boundsrosinski}, we get that 
	\[
	\Psi_{\Phi_{L}^{0}}(\xi_{n})\leq\int_{\mathbb{R}}\int_{\mathbb{R}^{d}}\mathbb{E}(V_{s}(q)^{2}\lor2)\Phi_{L}^{0}[f_{n}(s,q)]\mathrm{d}q\mathrm{d}s.
	\]
	All the conclusions of this proposition then follow easily by this,
	the $\mathcal{L}^{2}\left(\Omega,\mathcal{F},\mathbb{P}\right)$ boundedness
	of $V$ and Theorem \ref{DCTstcint}. \end{proof}
\begin{remark}\label{multidimensionalint}Integration of $\mathbb{R}^{m}$-valued
	predictable fields can be done entry by entry, that is, if $\xi=(\xi_{i})_{i=1}^{n}$,
	in which $\xi_{i}\in\mathcal{D}\left(I_{L}\right)$, then we define
	\[
	\int_{\mathbb{R}}\int_{\mathbb{R}^{d}}\xi(s,q)L(\mathrm{d}q\mathrm{d}s):=\left\{ \int_{\mathbb{R}}\int_{\mathbb{R}^{d}}\xi_{i}(s,q)L(\mathrm{d}q\mathrm{d}s)\right\} _{i=1}^{n}.
	\]
	Finally, observe that $\xi_{i}\in\mathcal{D}\left(I_{L}\right)$ is
	equivalent to 
	\[
	\int_{\mathbb{R}}\int_{\mathbb{R}^{d}}\Phi_{L}^{0}(\left\Vert \xi\left(s,q\right)\right\Vert )\mathrm{d}q\mathrm{d}s<\infty.
	\]
	In particular, when $\xi$ is deterministic, we have that the random
	vector $X=\int_{\mathbb{R}}\int_{\mathbb{R}^{d}}\xi(s,q)L(\mathrm{d}q\mathrm{d}s)$
	is ID and has characteristic triplet $(\Gamma_{X},B_{X},\nu_{X})$
	relative to some $\mathbb{R}^{m}$-valued truncation function $\stackrel{\rightarrow}{\tau}$,
	given by 
	\begin{align*}
	\Gamma_{X} & =\gamma_{\tau}\int_{\mathbb{R}}\int_{\mathbb{R}^{d}}\xi(s,q)\mathrm{d}q\mathrm{d}s+\int_{\mathbb{R}}\int_{\mathbb{R}^{d}}\int_{\mathbb{R}}[\stackrel{\rightarrow}{\tau}(x\xi(s,q))-\xi(s,q)\tau(x)\nu(\mathrm{d}x)]\mathrm{d}q\mathrm{d}s;\\
	B_{X} & =b^{2}\int_{\mathbb{R}}\int_{\mathbb{R}^{d}}\xi(s,q)\xi(s,q)'\mathrm{d}q\mathrm{d}s;\\
	\nu_{X}(A) & =\int_{\mathbb{R}}\int_{\mathbb{R}^{d}}\int_{\mathbb{R}}\mathbf{1}_{A}(x\xi(s,q))\nu(\mathrm{d}x)\mathrm{d}q\mathrm{d}s.
	\end{align*}	
\end{remark}

\section*{Appendix B}

This supplementary appendix contains several technical results that are used through the proofs of Theorems \ref{gaussianattractor}-\ref{stableattractor}.\label{key}
\subsection*{A Stochastic Fubini Theorem}

By using the L\'evy-It\^o decomposition of ID fields introduced
in \cite{Rosinski16}, cf. \cite{Ped03}, and a small refinement in
the arguments of Theorem 3.1 in \cite{BasseBN11}, cf. Lemma 4.9 in
\cite{BassePed09}, we obtain a stochastic Fubini's Theorem for surface
integrals and L\'evy bases. We recall that a random field $(Z(t))_{t\in T}$
is said to be separable in probability if there exist $T_{0}\subseteq T$
countable such that for any $t\in T$ it is possible to extract $\{t_{n}\}\subset T_{0}$
satisfying that $Z(t_{n})\overset{\mathbb{P}}{\rightarrow}Z_{t}.$		
\begin{lemma}\label{Fubinilemma}Fix $n\leq m$. Let $K\subset\mathbb{R}^{n}$
	be compact and $\varphi:K\rightarrow\mathbb{R}^{m}$ a continuously
	differentiable function with Jacobian $D\varphi$. Given a measurable
	function $f:\mathbb{R}^{m}\times\mathbb{R}^{d}\rightarrow\mathbb{R}$ and a homogeneous L\'evy basis $L$ with characteristic triplet $(\gamma,b,\nu)$,
	assume that the ID field
	\begin{equation}
	Z(p)=\int_{\mathbb{R}^{d}}f(p,q)L(\mathrm{d}q),\,\,\,p\in\mathbb{R}^{m},\label{Zdef}
	\end{equation}
	is well defined and separable in probability on $\varphi(K)$. Suppose
	in addition that
	\begin{enumerate}
		\item 	$\int_{K}\left[\int_{\mathbb{R}^{d}}\{\left|\gamma f(\varphi(u),q)\right|+\left|bf(\varphi(u),q)\right|^{2}\}\mathrm{d}q\right]\left|D\varphi(u)\right|\mathrm{d}u  <\infty$;
		\item $\int_{K}\left[\int_{\mathbb{R}^{d}}\int_{\left|x\right|\leq1}\left|xf(\varphi(u),q)\right|\land\left|xf(\varphi(u),q)\right|^{2}\nu(\mathrm{d}x)\mathrm{d}q\right]\left|D\varphi(u)\right|\mathrm{d}u <\infty$;
		\item $ \int_{\mathbb{R}^{d}}\int_{\left|x\right|>1}(\left|x\chi(q)\right|\land1)\nu(\mathrm{d}x)\mathrm{d}q  <\infty$,
	\end{enumerate}
	where $\left|D\varphi\right|:=Det(D\varphi^{'}D\varphi)^{1/2}$ and
	$\chi(q):=\int_{K}\left|f(\varphi(u),q)\right|\left|D\varphi(u)\right|\mathrm{d}u$.
	Then the random field $\{Z(\varphi(u))\}_{u\in K}$ can be chosen measurable
	and 
	\begin{align}
	\int_{K}Z(\varphi(u))\left|D\varphi(u)\right|\mathrm{d}u & =\int_{\mathbb{R}^{d}}\int_{K}f(\varphi(u),q)\left|D\varphi(u)\right|\mathrm{d}uL(\mathrm{d}q),\label{eq:fubini}
	\end{align}
	in the sense that both integrals exist and are equal almost surely.\end{lemma}		
\begin{proof}			
	Arguing as in \cite{BasseBN11}, we can always choose a measurable
	modification of $Z$, meaning that $\{Z(\varphi(u))\}_{u\in K}$ can
	be assumed to be measurable. Now, by the L\'evy-It\^o decomposition for
	L\'evy bases (see \cite{Ped03}) $L$ can be written as 
	\[
	L(A)=\gamma Leb(A)+W(A)+M(A)+J(A),\,\,\,A\in\mathcal{B}_{b}(\mathbb{R}^{d}),
	\]
	where $W,M\text{ and }J$ are three independent homogeneous L\'evy bases
	with characteristic triplets $(0,b,0)$, $(0,0,\left.\nu\right|_{[-1,1]})$
	and $(0,0,\left.\nu\right|_{[-1,1]^{c}})$, respectively, where the
	notation $\left.\nu\right|_{B}$ represents the restriction of the
	measure $\nu$ to the set $B$. Thus, $Z$ can be written as 
	\begin{equation}
	Z=Z^{\gamma}+Z^{W}+Z^{M}+Z^{J},\label{levyitoZ}
	\end{equation}
	where \textbf{$Z^{\gamma}$}, $Z^{W}$, $Z^{M}$ and $Z^{J}$ are
	independent fields defined as in (\ref{Zdef}) when we replace $L$
	by $\gamma Leb$,$W,M\text{ and }J$, respectively. Thus, it is enough
	to verify that (\ref{eq:fubini}) holds when we replace $Z$ by any
	of these fields. Now, it is clear that due to 1.,
	the usual Fubini's Theorem can be applied to $Z^{\gamma}$. Furthermore,
	since for every $A\in\mathcal{B}_{b}(\mathbb{R}^{d})$, $\mathbb{E}[\left|W(A)\right|]<\infty$,
	$\mathbb{E}[\left|M(A)\right|]<\infty$, we have that $W$ and $M$
	are within the framework of Theorem 3.1 in \cite{BasseBN11}. Consequently,
	in view that $\int_{K}\left|D\varphi(x)\right|\mathrm{d}x<\infty$,
	and 1. and 2. are
	satisfied, we have that the stochastic Fubini's Theorem in \cite{BasseBN11}
	can be applied to get that (\ref{eq:fubini}) is fulfilled for $Z^{W}$
	and $Z^{M}$. Therefore, it only remains to show that (\ref{eq:fubini})
	holds for $Z^{J}$. Since $Z$ is separable we may assume that the
	same holds for $Z^{J}$, which by Theorem 3.2 in \cite{Rosinski16}
	implies that for every $p\in\mathbb{R}^{m}$, almost surely 
	\[
	Z^{J}(p)=\int_{\mathbb{R}^{d}}\int_{\mathbb{R}}f(p,q)xN(\mathrm{d}x,\mathrm{d}q),
	\]
	where $N$ is a Poisson random measure with intensity $\mu(\mathrm{d}x,\mathrm{d}q)=\left.\nu\right|_{[-1,1]^{c}}(\mathrm{d}x)\otimes\mathrm{d}q$.
	Since $N(\mathrm{d}x,\mathrm{d}q)(\omega)$ is a $\sigma$-finite
	measure for every $\omega\in\Omega$, then $\int_{\mathbb{R}^{d}}\int_{\mathbb{R}}f(p,q)xN(\mathrm{d}x,\mathrm{d}q)$
	can be understood as a Lebesgue integral $\omega$ by $\omega$. Thus,
	to finish the proof, it is sufficient to show that on a set of probability
	one, it holds that
	\begin{equation}
	\int_{K}\int_{\mathbb{R}^{d}}\int_{\mathbb{R}}\left|f(\varphi(u),q)x\right|N(\mathrm{d}x,\mathrm{d}q)(\omega)\left|D\varphi(u)\right|\mathrm{d}u<\infty,\label{finitefubini}
	\end{equation}
	because in this case the usual Fubini's Theorem can be applied $\omega$
	by $\omega$ to obtain (\ref{eq:fubini}). By Tonelli's Theorem, almost
	surely
	\[
	\int_{K}\int_{\mathbb{R}^{d}}\int_{\mathbb{R}}\left|f(\varphi(u),q)x\right|N(\mathrm{d}x,\mathrm{d}q)\left|D\varphi(u)\right|\mathrm{d}u=\int_{\mathbb{R}^{d}}\int_{\mathbb{R}}\chi(q)\left|x\right|N(\mathrm{d}x,\mathrm{d}q).
	\]
	Equation (\ref{finitefubini}) follows from the previous equation,
	3. and Lemma 12.13 in \cite{Kallenberg02}.
\end{proof}		
\begin{remark}\label{remarkHausdorrff}Observe that under the assumptions
	of the previous lemma, if $\varphi$ is one-to-one (up to a null set)
	the area formula (see \cite{EvansGar92}) shows that almost surely
	\[
	\int_{K}Z(\varphi(u))\left|D\varphi(u)\right|\mathrm{d}u=\int_{C}Z(y)\mathcal{H}^{n}(\mathrm{d}y)=\int_{\mathbb{R}^{d}}\int_{C}f(y,q)\mathcal{H}^{n}(\mathrm{d}y)L(\mathrm{d}q),
	\]
	where $\mathcal{H}^{n}$ is the $n$-dimensional Hausdorff measure
	on $\mathbb{R}^{m}$.			
\end{remark}
\subsection*{Line integrals and Jordan domains}

In this part we deal with the asymptotic behaviour of certain type of line integrals. Recall the notation $(x,y)^{\perp}=(-y,x)$.
\begin{lemma}\label{keylemmaasymptotics}Let $A\subseteq\mathbb{R}^{2}$ be a Jordan domain with Lipschitz-regular boundary and $F:\mathbb{R}^{2}\rightarrow\mathbb{R}^{2}$ a measurable function that is continuous on $B^i\cap(\partial A)_{\oplus r_{0}^i}$ for some $r_{0}^i>0$, where for $i=1,2$, $B^{1}=A$ and $B^{2}=A^{*}$. Up to a null set, define for almost all $q\in\partial A$ 
	\begin{align}
	G_{r,\mathscr{D}}^{i}(s,q;F,A)&:=\ointop_{ r\mathbb{S}^1(q+rsu_{A}(q))}F\mathbf{1}_{B^{i}}\cdot n\mathrm{d}x;\\\label{Gfunctionkeylemma}
	G_{r,\mathscr{C}}^{i}(s,q;F,A)&:=\ointop_{ r\mathbb{S}^1(q+rsu_{A}(q))}F(u)\mathbf{1}_{B^{i}}\cdot\mathrm{d}u,
	\end{align}
	where $u_A(q)$ it the outward vector at $q$. Then 
	\begin{description}
		\item [{i)}] We have that 
		\[
		\sup_{r\leq r_{0},(s,q)\in[-1,1]\times\partial A}\left|r^{-1}G_{r,\mathscr{D}}^{i}(s,q;F,A)\right|<\infty;\,\,\,\,\sup_{r\leq r_{0},(s,q)\in[-1,1]\times\partial A}\left|r^{-1}G_{r,\mathscr{C}}^{i}(s,q;F,A)\right|<\infty,
		\]
		and for $\mathcal{H}^1$-almost all $q\in\partial A$, as $r\downarrow0$ 
		\begin{align}
		r^{-1}G_{r,\mathscr{D}}^{i}(s,q;F,A) & \rightarrow(-1)^{i}2\sqrt{1-\left|s\right|^{2}}F(q)\cdot u_{A}(q);\label{eq:convgDFnozero}\\
		r^{-1}G_{r,\mathscr{C}}^{i}(s,q;F,A) & \rightarrow(-1)^{i}2\sqrt{1-\left|s\right|^{2}}F(q)\cdot v_{A}(q).\label{eq:convgCFnozero}
		\end{align}
		\item [{ii)}] If in addition $F$ is continuously differentiable on $B^{i}\cap(\partial A)_{\oplus r_{0}}$
		and$\left.F\right|_{\partial A}=0$, then
		\[
		\sup_{r\leq r_{0},(s,q)\in[-1,1]\times\partial A}\left|r^{-2}G_{r,\mathscr{D}}^{i}(s,q;F,A)\right|<\infty;\,\,\,\,\sup_{r\leq r_{0},(s,q)\in[-1,1]\times\partial A}\left|r^{-2}G_{r,\mathscr{C}}^{i}(s,q;F,A)\right|<\infty,
		\]
		Furthermore, for $\mathcal{H}^1$-almost all $q\in\partial A$ and $\left|s\right|<1$,
		the following limits hold as $r\downarrow0$ 
		\begin{align}
		r^{-2}G_{r,\mathscr{D}}^{i}(s,q;F,A) & \rightarrow\begin{cases}
		\left[DF(q)_{11}+DF(q)_{22}\right](\pi+\left|s\right|\sqrt{1-s^{2}}-\arccos(\left|s\right|)) & \text{if }i=1;\\
		\left[DF(q)_{11}+DF(q)_{22}\right](\arccos(\left|s\right|)-\left|s\right|\sqrt{1-s^{2}}) & \text{if }i=2;
		\end{cases}\label{eq:convgD-1}\\
		r^{-2}G_{r,\mathscr{C}}^{i}(s,q;F,A) & \rightarrow\begin{cases}
		[DF(q)_{21}-DF(q)_{12}](\pi+\left|s\right|\sqrt{1-s^{2}}-\arccos(\left|s\right|)) & \text{if }i=1;\\{}
		[DF(q)_{21}-DF(q)_{12}](\arccos(\left|s\right|)-\left|s\right|\sqrt{1-s^{2}}) & \text{if }i=2,
		\end{cases}\label{eq:convgC-1}
		\end{align}
		where $DF$ denotes the Jacobian of $F$.
	\end{description}
\end{lemma}
\begin{proof}We will only concentrate on $G_{r,\mathscr{D}}^{i}$. Observe that $G_{r,\mathscr{D}}^{i}$ and
	is well defined and measurable for any $r\leq r_{0}$. Now, by the continuity
	of $F$, we have that
	\[
	\left|G_{r,\mathscr{D}}^{i}(s,q;F,A)\right| \leq 2\pi r\sup_{ B^{i}\cap(\partial A)_{\oplus r_{0}^i}}\left\Vert F(q)\right\Vert <\infty,
	\]
	which is the first conclusion in i). On the other hand, $F$ is continuously
	differentiable on $B^{i}\cap(\partial A)_{\oplus r_{0}^i}$ and$\left.F\right|_{\partial A}=0$,
	then by the Mean Value Theorem, we have that for $\mathcal{H}^1$-a.a. $q\in\partial A$
	and $\left|s\right|\leq1$
	\begin{align*}
	\left|G_{r,\mathscr{D}}^{i}(s,q;F,A)\right| & \leq r\int_{0}^{2\pi}\left| F(q+rsu_{A}(q)+ru(\theta))-F(q)\cdot u(\theta) \right|\mathrm{d}\theta\\
	& =r^{2}\int_{0}^{2\pi}\left| A_{r}(q,s,\theta)[su_{A}(q)+u(\theta)]\cdot u(\theta)\right|\mathrm{d}\theta,
	\end{align*}
	where $A_{r}(q,s,\theta)=\int_{0}^{1}DF(q+rx[su_{A}(q)+u(\theta)])\mathrm{d}x$
	with $DF$ the Jacobian matrix of $F$. Since $F$ is continuously
	differentiable on $B^{i}\cap(\partial A)_{\oplus r_{0}}$ and $u_{A}(q)$
	and $u(\theta)$ are unitary, we get
	\[
	\left|G_{r,\mathscr{D}}^{i}(s,q;F,A)\right|\leq r^{2}4\pi\sup_{B^{i}\cap(\partial A)_{\oplus r_{0}}}\left\Vert DF(q)\right\Vert <\infty.
	\]
	The first part of $ii)$ is obtained from this. 
	\begin{figure}[h]
		\center
		\includegraphics[width=0.5\textwidth]{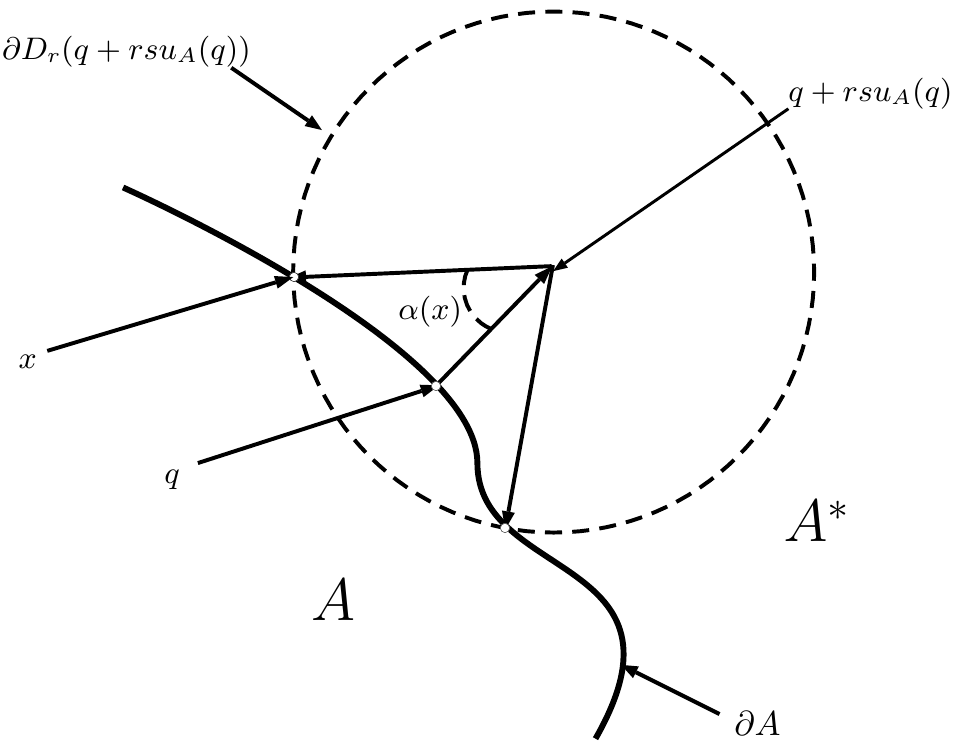}
		\caption{\label{fig:sub1A} The figure illustrates a typical shape of $ r\mathbb{S}^1(q+rsu_{A}(q))\cap\partial A$}
	\end{figure}	
	In what follows, we fix $q\in\partial A$ from which the $u_A(q)$ is well defined and all the limits
	appearing below are taken when $r\downarrow0$. If $s=\pm1$, we have
	that for $r$ small enough, $ r\mathbb{S}^1(q+rsu_{A}(q))\cap B^{i}=\{q\}$ for any $i=1,2$.
	In this case (\ref{eq:convgDFnozero}) and (\ref{eq:convgCFnozero})
	follow trivially, so for the rest of the proof we consider only the
	case when $\left|s\right|<1$. Assume first that $s>0$, $\left.F\right|_{\partial A}\neq0$
	and $i=2$. Let $\phi_{q}$ be the angle of $u_{A}(q)$ and parametrize
	$ r\mathbb{S}^1(q+rsu_{A}(q))$ as
	\[
	\varphi_{q,r,s}(\theta):=q+rsu_{A}(q)+ru(\theta+\phi_{q}),\,\,\,0\leq\theta<2\pi,
	\]
	in such a way that
	\[
	G_{r,\mathscr{D}}^{2}(s,q;F,A)=r\int_{_{\Theta_{r}}}F(\varphi_{q,r,s}(\theta))\cdot u(\theta+\phi_{q})\mathrm{d}\theta,
	\]
	where $\Theta_{r}=\{0\leq\theta<2\pi:\varphi_{q,r,s}(\theta)\in A^{*}\}$.
	We note that $\Theta_{r}$ is based on arcs between the elements of
	$ r\mathbb{S}^1(q+rsu_{A}(q))\cap\partial A$, so we now proceed to
	describe it.	
	Let $\theta(x)$ be such that $x=\varphi_{q,r,s}(\theta(x)).$ The
	mapping $x\mapsto\theta(x)$ is well defined and continuous on the closed set $ r\mathbb{S}^1(q+rsu_{A}(q))\cap\partial A$.
	Indeed, if we denote by $\alpha(x)$ the angle between $rsu_{A}(q)$
	and $x-(q+rsu_{A}(q))$, see Figure \ref{fig:sub1A}, we get by the
	the cosine law and Taylor's Theorem that
	\begin{equation}
	\theta(x)=\pi\pm\alpha(x)=\pi\pm\arccos(s-\frac{x-q}{r}\cdot u_{A}(q)),\,\,\,x\in r\mathbb{S}^1(q+rsu_{A}(q))\cap\partial A,\label{thetaxdef}
	\end{equation}
	according to whether $\theta(x)\in[\pi,2\pi)$ or $\theta(x)\in[0,\pi)$.
	The continuity then follows from this. Hence, we have that there are
	$x_{max,r}^{+},x_{min,r}^{+},x_{max,r}^{-},x_{min,r}^{-}\in r\mathbb{S}^1(q+rsu_{A}(q))\cap\partial A$
	such 
	\begin{align*}
	\theta(x_{min,r}^{-})=\inf\{\theta\in[0,\pi]:\varphi_{q,r,s}(\theta)\in\partial A\}; & \,\,\,\theta(x_{max,r}^{-})=\sup\{\theta\in[0,\pi]:\varphi_{q,r,s}(\theta)\in\partial A\},\\
	\theta(x_{min,r}^{+})=\inf\{\theta\in[\pi,2\pi]:\varphi_{q,r,s}(\theta)\in\partial A\}; & \,\,\,\theta(x_{max,r}^{+})=\sup\{\theta\in[\pi,2\pi]:\varphi_{q,r,s}(\theta)\in\partial A\},
	\end{align*}
	and $0\leq\theta(x_{min,r}^{-})\leq\theta(x_{max,r}^{-})\leq\pi\leq\theta(x_{min,r}^{+})\leq\theta(x_{max,r}^{+})\leq2\pi$.
	Now suppose for a moment that $r^{-1}\left\Vert x_{min,r}^{\pm}-q\right\Vert $
	and $r^{-1}\left\Vert x_{max,r}^{\pm}-q\right\Vert $ are bounded
	over $r$. If this were true, then by the Lipschitz-regularity condition on $\partial A$, we would have $\theta(x_{min,r}^{\pm})\rightarrow\pi\pm\arccos(s)$
	and $\theta(x_{max,r}^{\pm})\rightarrow\pi\pm\arccos(s)$ which, together
	with the fact that $u_{A}(q)=u(\phi_{q})$, would give us that
	\begin{align*}
	r^{-1}G_{r,\mathscr{D}}^{2}(s,q;F,A) & =\int_{-\theta(x_{max,r}^{+})}^{\theta(x_{min,r}^{-})}F(\varphi_{q,r,s}(\theta))\cdot u(\theta+\phi_{q})\mathrm{d}\theta\\
	& +\int_{\Theta_{r}\cap[\theta(x_{min,r}^{+}),\theta(x_{max,r}^{+})]}F(\varphi_{q,r,s}(\theta))\cdot u(\theta+\phi_{q})\mathrm{d}\theta\\
	& +\int_{\Theta_{r}\cap[\theta(x_{min,r}^{-}),\theta(x_{max,r}^{-})]}F(\varphi_{q,r,s}(\theta))\cdot u(\theta+\phi_{q})\mathrm{d}\theta\\
	& \rightarrow\int_{\arccos(s)-\pi}^{\pi-\arccos(s)}F(q)\cdot u(\theta+\phi_{q})\mathrm{d}\theta=2\sqrt{1-s^{2}}F(q)\cdot u_{A}(q),
	\end{align*}
	which is (\ref{eq:convgDFnozero}). We will only check that $r^{-1}\left\Vert x_{min,r}^{-}-q\right\Vert $
	is bounded since the boundness of the other quantities can be shown
	in an similar way. To see the former holds, let $\gamma(x_{min,r}^{-})$
	be as in Figure \ref{fig:sub2A} and note that by the cosine law
	\[
	2r\left\Vert x_{min,r}^{-}-q\right\Vert \cos(\gamma(x_{min,r}^{-}))=\left\Vert x_{min,r}^{-}-q\right\Vert ^{2}+r^{2}(1-s)\geq\left\Vert x_{min,r}^{-}-q\right\Vert ^{2}>0,
	\]
	as claimed. 
	\begin{figure}[h] 
		\center
		\includegraphics[width=0.5\textwidth]{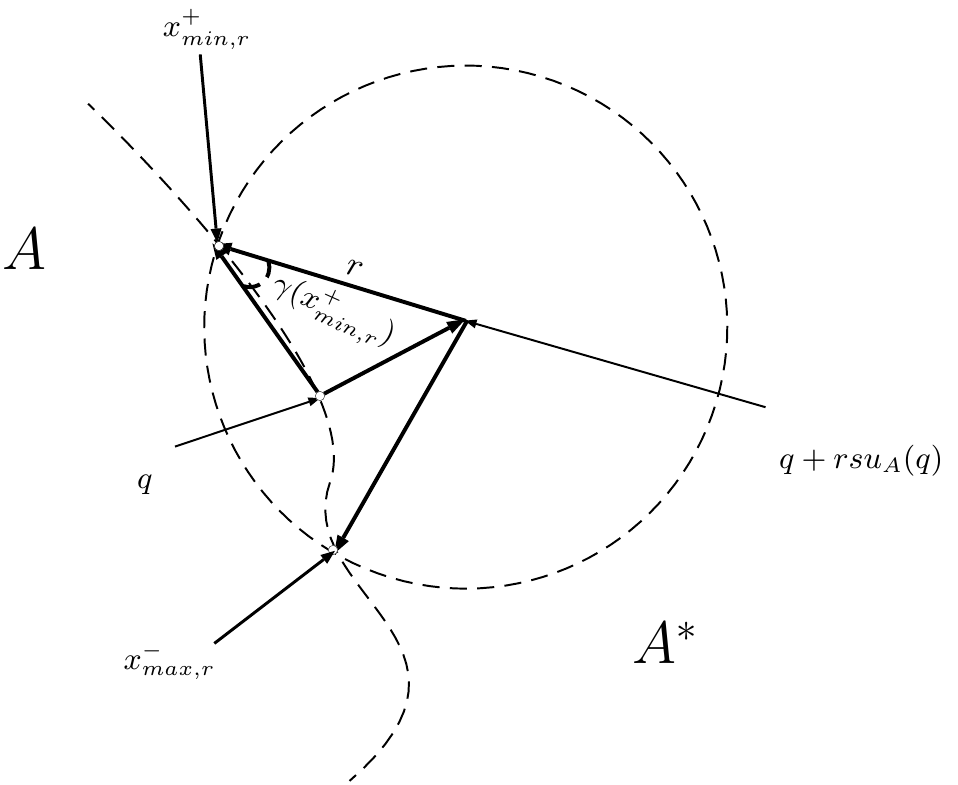}
		\caption{\label{fig:sub2A} The figure shows the construction of the angle $\gamma(x_{min,r}^{-})$.} 
	\end{figure}
	The case $i=1$ follows by noting that
	\begin{align}
	r^{-1}G_{r,\mathscr{D}}^{1}(s,q;F,A) & =r^{-1}\left\{ \ointop_{ r\mathbb{S}^1(q+rsu_{A}(q))}F\cdot n\mathrm{d}x-G_{r,\mathscr{D}}^{2}(s,q;F,A)\right\} \rightarrow-2\sqrt{1-s^{2}}F(q)\cdot u_{A}(q).\label{Intlimitnon-vanis}
	\end{align}
	For the situation when $-1<s<0$, we observe that
	\begin{equation}
	G_{r,\mathscr{D}}^{i}(s,q;F,A)=\ointop_{ r\mathbb{S}^1(q+r\left|s\right|u_{A^{*}}(q))\cap B^{i}}F\cdot n\mathrm{d}x,\label{eq:sneg}
	\end{equation}
	where $u_{A^{*}}(q)$ is the outer vector to $A^{*}$ at $q$. Thus,(\ref{eq:convgCFnozero}) follows by replacing $u_{A}(q)$ by $u_{A^{*}}(q)$ in the preceding arguments. This concludes the proof of (\ref{eq:convgDFnozero}).
	The proof of (\ref{eq:convgCFnozero}) is done by changing $u$ by
	$u^{\perp}$ in the previous reasoning. 	
	Now suppose that $\left.F\right|_{\partial A}=0$, $i=2$ and $1>s>0$.
	Under the notation used above we get that in this case
	\begin{align*}
	r^{-2}G_{r,\mathscr{D}}^{2}(s,q;F,A) & =\int_{-\theta(x_{max,r}^{+})}^{\theta(x_{min,r}^{-})}A_{r}(q,s,\theta)[su_{A}(q)+u(\theta+\phi_{q})]\cdot u(\theta+\phi_{q})\mathrm{d}\theta\\
	& +\int_{\Theta_{r}\cap[\theta(x_{min,r}^{+}),\theta(x_{max,r}^{+})]}A_{r}(q,s,\theta)[su_{A}(q)+u(\theta+\phi_{q})]\cdot u(\theta+\phi_{q})\mathrm{d}\theta\\
	& +\int_{\Theta_{r}\cap[\theta(x_{min,r}^{-}),\theta(x_{max,r}^{-})]}A_{r}(q,s,\theta)[su_{A}(q)+u(\theta+\phi_{q})]\cdot u(\theta+\phi_{q})\mathrm{d}\theta\\
	& \rightarrow\int_{\arccos(s)-\pi}^{\pi-\arccos(s)}DF(q)[su_{A}(q)+u(\theta+\phi_{q})]\cdot u(\theta+\phi_{q})\mathrm{d}\theta\\
	& =\left[DF(q)_{11}+DF(q)_{22}\right](\pi+s\sqrt{1-s^{2}}-\arccos(s)).
	\end{align*}	
	Since
	\begin{align*}
	r^{-2}\ointop_{ r\mathbb{S}^1(q+rsn_{A}(q))}F\cdot n\mathrm{d}x&=\int_{0}^{2\pi}A_{r}(q,s,\theta)[su_{A}(q)+u(\theta+\phi_{q})]\cdot u(\theta+\phi_{q})\mathrm{d}\theta\\
	&\rightarrow\pi[DF(q)_{11}+DF(q)_{22}],
	\end{align*}
	we can then use a similar argument as in (\ref{Intlimitnon-vanis})
	to show that (\ref{eq:convgD-1}) holds when $i=1$ and $s\geq0$.
	Moreover, due to equation (\ref{eq:sneg}), the case $s<0$ in (\ref{eq:convgD-1})
	can be obtained using our former reasoning. Finally, to complete the argument of the proof, we note that the interchange
	of all the previous limits with the integral sign is possible due
	to the uniform bounds shown at the beginning of the proof and the
	Dominated Convergence Theorem.\end{proof}
\begin{remark}\label{remarkO(1)}Arguing as in the
	proof of the previous lemma, it is possible to show that if $F$ is a vector-valued function that is continuous on
	$B^{i}\cap(\partial A)_{\oplus r_{0}}$ for some $r_{0}>0$, then
	\[
	\sup_{r\leq r_{0,}q\in(\partial A)_{\oplus r}}\left|r^{-1}\ointop_{ r\mathbb{S}^1(q)}F\mathbf{1}_{B^{i}}\cdot n\mathrm{d}x\right|<\infty,\,\,\,\text{and}\,\,\,\sup_{r\leq r_{0,}q\in(\partial A)_{\oplus r}}\left|r^{-1}\ointop_{ r\mathbb{S}^1(q)}F(u)\mathbf{1}_{B^{i}}\cdot\mathrm{d}u\right|<\infty.
	\]
	If in addition $F$ is continuously differentiable on $B^{i}\cap(\partial A)_{\oplus r_{0}}$
	and$\left.F\right|_{\partial A}=0$, then
	\[
	\sup_{r\leq r_{0,}q\in(\partial A)_{\oplus r}}\left|r^{-2}\ointop_{ r\mathbb{S}^1(q)}F\mathbf{1}_{B^{i}}\cdot n\mathrm{d}x\right|<\infty,\,\,\,\text{and}\,\,\,\sup_{r\leq r_{0,}q\in(\partial A)_{\oplus r}}\left|r^{-2}\ointop_{ r\mathbb{S}^1(q)}F(u)\mathbf{1}_{B^{i}}\cdot\mathrm{d}u\right|<\infty.
	\]
\end{remark}
\subsection*{An approximation}
The following lemma allows to replace $L$ in the definitions of $\mathscr{C}_{r}$
and $\mathcal{\mathscr{D}}_{r}$ by a strictly $\beta$-stable homogeneous L\'evy basis.
\begin{lemma}\label{keylemmaasymptotics-1}Let $A\subseteq\mathbb{R}^{2}$ be a Jordan domain with Lipschitz-regular boundary and $F:\mathbb{R}^{2}\rightarrow\mathbb{R}^{2}$ be such that $F(-\cdot)$ is as in Lemma \ref{keylemmaasymptotics}. Suppose that $\psi$ is the characteristic
	exponent of an ID distribution on $\mathbb{R}$ and that there exists
	$1\leq\beta\leq2$ for which $r\psi(r^{-1/\beta}\cdot)\rightarrow\psi_{\mu}(\cdot)$,
	as $r\downarrow0$, where $\psi_{\mu}$ is the characteristic exponent
	of an ID distribution $\mu$. Then $\mu$ is (possibly deformed) strictly
	$\beta$-stable. Furthermore, if $n\in\mathbb{N}$, $(p_{j})_{j=1}^{n}\subseteq\mathbb{R}^{2}$,
	$(z_{j})_{j=1}^{n}\subseteq\mathbb{R}$, then for $i=1,2$, as $r\downarrow0$
	\begin{align}
	\left|\int_{\cup_{j=1}^{n}(\partial A+p_{j})_{\oplus r}}\left[\psi(r^{-1/\beta}z_{r,d,n}^{i}(q))-\psi_{\mu}(r^{-1/\beta}z_{r,d,n}^{i}(q))\right]\mathrm{d}q\right| & \rightarrow0,\label{approxd}\\
	\left|\int_{\cup_{j=1}^{n}(\partial A+p_{j})_{\oplus r}}\left[\psi(r^{-1/\beta}z_{r,c,n}^{i}(q))-\psi_{\mu}(r^{-1/\beta}z_{r,c,n}^{i}(q))\right]\mathrm{d}q\right| & \rightarrow0,\label{approxc}
	\end{align}
	where 
	\begin{align*}
	z_{r,d,n}^{i}(q):= & r^{-1}\sum_{j=1}^{n}\mathbf{1}_{(\partial A+p_{j})_{\oplus r}}(q)z_{j}\ointop_{ r\mathbb{S}^1(q)}F(p_{j}-\cdot)\mathbf{1}_{B^{i}}\cdot n\mathrm{d}x;\\
	z_{r,c,n}^{i}(q):= & r^{-1}\sum_{j=1}^{n}\mathbf{1}_{(\partial A+p_{j})_{\oplus r}}(q)z_{j}\ointop_{ r\mathbb{S}^1(q)}F(p_{j}-u)\mathbf{1}_{B^{i}}\cdot\mathrm{d}u.
	\end{align*}
\end{lemma}	
\begin{proof}The fact that $\mu$ is strictly stable follows by Theorem
	1 in \cite{IvanovJ17}, cf. \cite{Lamp62}. Observe that the convergence
	$r\psi(r^{-1/\beta}\cdot)\rightarrow\psi_{\mu}(\cdot)$ can always
	be strengthen to uniform convergence on compacts. Furthermore, thanks
	to Remark \ref{remarkO(1)}, for any $i=1,2$, we can choose $a<b$
	in such a way that $b\leq z_{r,d,}^{i}(q)\leq a$, for every $r\leq r_{0}^{i}$
	and $q\in\cup_{j=1}^{n}(\partial A+p_{j})_{\oplus r}$. Thus
	\begin{align*}
	\int_{\cup_{j=1}^{n}(\partial A+p_{j})_{\oplus r}}\left|\psi(r^{-1/\beta}z_{r,d,n}^{i}(q))-\psi_{\mu}(r^{-1/\beta}z_{r,d,n}^{i}(q))\right|&\mathrm{d}q\leq n\frac{Leb((\partial A)_{\oplus r})}{r}\\
	&\times\sup_{a\leq u\leq b}r\left|\psi(r^{-1/\beta}u)-\psi_{\mu}(r^{-1/\beta}u)\right|
	\end{align*}
	Since $\partial A$ is Lipschitz, by Theorem 5 and Corollary 1 in \cite{AmbrosioColVil08},
	we get that as $r\downarrow0$, $\frac{Leb((\partial A)_{\oplus r})}{r}\rightarrow2\mathcal{H}^{1}(\partial A)<\infty$.
	A combination of this and the previous estimate, give us (\ref{approxd}).
	The approximation in (\ref{approxc}) is shown in a similar way.		
\end{proof}	

\bibliographystyle{plain}
\bibliography{bibSept17}

\end{document}